\documentclass[11pt,a4paper]{amsart}

\usepackage{times,epsfig}
\usepackage{amsfonts,amscd,amsmath,amssymb,amsthm}

\usepackage{tikz}
\usepackage{xy} \xyoption{all}

\usepackage{color,graphics}

\usepackage{bbm}

\usepackage{hyperref}

\numberwithin{equation}{section}

\newcommand{\uloopr}[1]{\ar@'{@+{[0,0]+(-4,5)}@+{[0,0]+(0,10)}@+{[0,0] +(4,5)}}^{#1}}
\newcommand{\uloopd}[1]{\ar@'{@+{[0,0]+(5,4)}@+{[0,0]+(10,0)}@+{[0,0]+ (5,-4)}}^{#1}}
\newcommand{\dloopr}[1]{\ar@'{@+{[0,0]+(-4,-5)}@+{[0,0]+(0,-10)}@+{[0, 0]+(4,-5)}}_{#1}}
\newcommand{\dloopd}[1]{\ar@'{@+{[0,0]+(-5,4)}@+{[0,0]+(-10,0)}@+{[0,0 ]+(-5,-4)}}_{#1}}

\newcommand{\luloop}[1]{\ar@'{@+{[0,0]+(-8,2)}@+{[0,0]+(-10,10)}@+{[0, 0]+(2,2)}}^{#1}}

\DeclareSymbolFont{SY}{U}{psy}{m}{n}
\DeclareMathSymbol{\emptyset}{\mathord}{SY}{'306}

\DeclareMathSymbol{\newtimes}{\mathbin}{SY}{'264}

\newcommand{\R}{\mathbb{R}}

\newcommand{\C}{\mathbb{C}}
\newcommand{\K}{\mathbb{K}}
\newcommand{\Z}{\mathbb{Z}}
\newcommand{\N}{\mathbb{N}}

\newcommand{\cA}{{\mathcal A}}
\newcommand{\cB}{{\mathcal B}}

\newcommand{\cF}{{\mathcal F}}

\newcommand{\cP}{{\mathcal P}}

\newcommand{\ol}{\overline}

\renewcommand{\1}{\mathbbm 1}

{\bf}{\it}
{\bf}{\it}
{\bf}{\it}
{\bf}{\it}

{\bf}{\it}

\newtheorem{theorem}{Theorem}[section]{\bf}{\it}
\newtheorem{proposition}[theorem]{Proposition}{\bf}{\it}
\newtheorem{corollary}[theorem]{Corollary}{\bf}{\it}
\newtheorem{lemma}[theorem]{Lemma}{\bf}{\it}
{\bf}{\it}
{\bf}{\it}
{\bf}{\it}

\theoremstyle{definition}
\newtheorem{definition}[theorem]{Definition}
\newtheorem{remark}[theorem]{Remark}
\newtheorem{example}[theorem]{Example}

\theoremstyle{plain}
\newcounter{theoremintro}
\newtheorem{thmintro}[theoremintro]{Theorem}
\theoremstyle{definition}
\newtheorem{defnintro}[theoremintro]{Definition}

\DeclareMathAlphabet{\Ma}{U}{msa}{m}{n}
\DeclareMathAlphabet{\Mb}{U}{msb}{m}{n}

\DeclareMathAlphabet{\Meuf}{U}{euf}{m}{n}
\DeclareSymbolFont{ASMa}{U}{msa}{m}{n}
\DeclareSymbolFont{ASMb}{U}{msb}{m}{n}
\DeclareMathSymbol{\hrist}{\mathord}{ASMa}{"16}
\DeclareMathSymbol{\varkappa}{\mathalpha}{ASMb}{"7B}
\DeclareMathSymbol{\CrPr}{\mathord}{ASMb}{"6F}

\def\got#1{\Meuf{#1}}
\def\ot #1.{{\got{#1}}}

\title{Amenability of coarse spaces and $\mathbb{K}$-algebras}

\author[Pere Ara]{Pere Ara$^{1}$}
\address{Department of Mathematics,
  Universitat Aut\`onoma de Barcelona, 08193 Bellaterra (Bar\-ce\-lona), Spain}
\email{para@mat.uab.cat}

\author[Kang Li]{Kang Li$^{2}$}
\address{Department of Mathematics, University of M\"{u}nster,
  Einsteinstr. 62, 48149 M\"{u}nster, Germany}
 \email{lik@uni-muenster.de} 
  
\author[Fernando Lled\'o]{Fernando Lled\'o$^{3}$}
\address{Department of Mathematics, University Carlos~III Madrid,
  Avda.~de la Universidad~30, 28911 Legan\'es (Madrid), Spain
  and Instituto de Ciencias Matem\'{a}ticas (CSIC - UAM - UC3M - UCM).}
\email{flledo@math.uc3m.es}

\author[Jianchao Wu]{Jianchao Wu$^{4}$}
\address{Department of Mathematics, Penn State University, 109 McAllister Building, University Park, PA 16802, USA}
\email{jianchao.wu@psu.edu}

\date{\today}
\thanks{{$^{1}$ } Supported by the grants DGI MICIIN MTM2011-28992-C02-01 and MINECO MTM2014-53644-P}
\thanks{{$^{2}$} Supported by ERC Advanced Grant no.~OAFPG 247321, the Danish National Research Foundation through the Centre for Symmetry and Deformation (DNRF92) and the Danish Council for Independent Research (DFF-5051-00037)}
\thanks{{$^{3}$} Supported by projects DGI MTM2012-36732-C03-01, MTM2014-54692-P and Severo Ochoa SEV-2015-0554 of the Spanish Ministry of Economy and Competition (MINECO), Spain}
\thanks{{$^{4}$} Supported by SFB 878 {\em Groups, Geometry and Actions} and ERC Advanced Grant ToDyRiC 267079}


\begin{document}

\begin{abstract}
In this article we analyze the notions of amenability and paradoxical decomposition from an 
algebraic perspective. We consider this dichotomy for locally finite extended metric 
spaces and for general algebras over fields. In the context of algebras we 
also study the relation of amenability with proper infiniteness. We apply our general analysis
to two important classes of algebras: the unital Leavitt path algebras and the translation 
algebras on locally finite extended metric spaces. In particular, we show that the amenability
of a metric space is equivalent to the algebraic amenability of the corresponding translation 
algebra.
\end{abstract}

\maketitle

\tableofcontents

\section{Introduction}\label{sec:intro}
Given a group $\Gamma$, von~Neumann defined in Section~1 of \cite{Neumann29} the notion of
{\em allgemeiner Mittelwert auf $\Gamma$} in terms of a mean (i.e., a finitely additive
probability measure) on $\Gamma$ which is left invariant under the action of $\Gamma$ on itself.
This property of the group eventually came to be called amenability \cite{Day57}. Its absence
was recognized by von~Neumann as a fundamental reason behind phenomena like the Banach-Tarski paradox \textemdash~ 
a paradoxical decomposition of the unit ball in $\R^3$. In fact, there is, for any group $\Gamma$, a complete dichotomy between amenability and the existence of paradoxical decompositions of $\Gamma$ in a natural sense, and the Banach-Tarski paradox may be essentially attributed to the fact that 
the (discrete) group SO$(3)$ of 
isometries of the ball contains a subgroup which is isomorphic to the free group 
$\mathbbm{F}_2$ on two generators, whose evident paradoxicality implies that of the former. By contrast, the group SO$(2)$ of isometries of the unit disc, like any other abelian group, is amenable, and thus not paradoxical. 
Later, F\o lner gave an equivalent characterization of 
amenability by the existence of a net $\{\Gamma_i\}_{i \in I}$ of non-empty
finite subsets of the group that, under the left translations of the group on itself, becomes more and more invariant in a statistical sense
(cf.,  \cite{Foelner55}). More precisely, one has that $\Gamma$ is amenable if and only if there exists
a net $\{\Gamma_i\}_{i \in I}$ of non-empty finite subsets with
\begin{equation*}\label{eq:foelner-group}
\lim_i\frac{|\gamma\Gamma_i \cup \Gamma_i|}{|\Gamma_i|}=1 \;,  
   \quad\text{for~any~} \gamma \in \Gamma \; ,
\end{equation*}
where $|\cdot|$ denotes the cardinality of the subset. These so-called F\o lner nets thus provide a good way to approximate an amenable infinite structure with finite substructures, opening the door to a wide range of applications. Moreover, thanks to its simplicity, F\o lner's characterization also lends itself to various generalizations, as we shall see below.
Since then, the concept of amenability has become central in many areas of 
mathematics like ergodic theory, geometry, the theory of operator algebras, etc.
Some classical references on this topic are \cite{Wagon94,bPaterson88,bRunde02}.

This paper studies amenability and paradoxical decompositions from an algebraic perspective. 
To provide a source of inspiration, we start with a review of amenability for metric spaces, a concept defined by Block and Weinberger in
\cite{Block-Weinberger-92} through a natural generalization of 
F\o lner's characterization to (uniformly) locally finite metric spaces \textemdash~similar ideas go as far back as the work of Ahlfors (\cite[II]{Ahlfors}) under the term \emph{Aussch\"{o}pfungen einer offenen Fl\"{a}che (exhaustions of an open surface)}. More precisely, a locally finite metric space $(X,d)$ is said to be amenable if there
exists a net $\{F_i\}_{i\in I}$ of finite non-empty subsets such that 
  \[
   \lim_{i} \frac{| N_R F_i |}{|F_i|} = 1 \;,  
   \quad\text{for~any~} R > 0 \;,
  \]
where $N_R F_i:= \{x\in X:d(x,F_i)\leq R\}$, the $R$-neighborhood of $F_i$ (cf., Definition~\ref{def:metric-amenability} and
Remark~\ref{rmk:charac-metric-amenability}). One of the key results in this setting, shown by Cecherini-Silberstein, Grigorchuk and de la Harpe in
\cite{Silberstein-Grigorchuk-Harpe-99}, states that in analogy with the well-known result for groups, 
the amenability of a metric space is equivalent to its non-paradoxicality, and also equivalent to the existence of an invariant mean, in a suitable sense 
(cf., Definition~\ref{def:invariant-mean} and Definition~\ref{def:paradoxical-decomposition}). For the convenience of the reader, we present a direct proof of the most interesting implication among them, namely that non-paradoxicality implies amenability, by adapting a proof in the group setting given in \cite{KL15} (cf., Theorem~\ref{theorem:amenable-extended-metric}). The key idea in it is a local-to-global technique that involves a variant of Hall's marriage theorem for sets of arbitrary cardinalities. A linearization of this technique will be applied later to prove a corresponding implication in the case of algebras over a field, which is the second main object of study in this article.

Let us fix a field $\K$. Elek introduced in \cite{Elek03} the notion of amenability for finitely generated unital algebras over $\K$, and proved some essential results in the case where the algebra has no zero-divisors. The main definition he used also resembles F\o lner's characterization, with subsets replaced by linear subspaces, and cardinalities replaced by dimensions. We generalize this notion to $\K$-algebras of arbitrary dimensions and single out a more restrictive situation brought about by the additional requirement that the F\o lner net is exhaustive, which we term \emph{proper amenability}.

\begin{defnintro}[cf., Definition~\ref{def:alg-amenable2} and Remark~\ref{rmk:charac-alg-amen2}]\label{def:alg-amenable-intro}
Let $\K$ be a field. An $\K$-algebra $\cA$ is said to be (left) \emph{algebraically amenable} if there exists a net $\{W_i\}_{i\in I}$ of finite-$\K$-dimensional linear subspaces of $\cA$ such that 
 \[
  \lim_{i} \frac{\dim_{\K}(a W_i +W_i)}{\dim_{\K}(W_i)} = 1 \;,  
   \quad\text{for~any~}  a\in\cA\;.
 \] 
 If the net $\{W_i\}_{i\in I}$ can be chosen to satisfy the additional condition that for any $a \in \cA$, there is $i\in I$ such that 
 \[
  a \in \bigcap_{j \geq i} W_j \; ,
 \]
 then $\cA$ is said to be (left) \emph{properly algebraically amenable}.
\end{defnintro}

Following Elek's pioneering work, a number of authors have dealt with amenability for algebras from
different perspectives, such as Bartholdi \cite{Bart}, Cecherini-Silberstein and Saimet-Vaillant \cite{Cec-Sam-08}, and D'Adderio \cite{D'adderio} (building on work 
of Gromov \cite{Gromov08}).
Special attention has been paid by Elek to the case of division algebras over a field, see \cite{Elek06,Elek16,Elekunp}. In particular, the notion of 
amenability for division algebras plays an important role in the study of infinite dimensional representations of a finite-dimensional algebra over a finite field undertaken in \cite{Elekunp}.

The fundamental result of Elek in \cite{Elek03} is the equivalence, for finitely generated unital $\K$-algebras without zero-divisors, among three
characterizations of algebraic amenability analogous to those in the cases of groups and metric spaces: algebraic amenability \`{a} la F\o lner as
given in Definition~\ref{def:alg-amenable-intro}, the non-existence of paradoxical decompositions, and an analogue of von~Neumann's invariant means
called invariant \emph{dimension measures}. The definitions of the latter two notions enlist the involvement of linear bases of the algebra. We offer
here generalizations of these notions (cf., Definition~\ref{def:paradoxicality} and Definition~\ref{def:dimension-measure}) and of Elek's theorem to encompass
all $\K$-algebras regardless of the size of the generating set or the existence of zero-divisors or a unit. Notably, invariant dimension measures in
our definition exhibit delicate deviations from von~Neumann's invariant means on a group, owing to the fact that the lattice of subspaces of an
algebra is not distributive, unlike the lattice of subsets of a group. For the sake of brevity, here we state the generalized theorem only for
countably dimensional $\K$-algebras. 
\begin{thmintro}[cf., Theorem~\ref{theorem:aa-characterized} and Corollary~\ref{cor:charac-amen-countable}]\label{theorem:aa-characterized-intro}
 Let $\cA$ be a countably dimensional $\K$-algebra over a field $\K$. Then the following are equivalent:
 \begin{enumerate}
  \item $\cA$ is algebraically amenable.
  \item There is a linear basis of $\cA$ that cannot be paradoxically decomposed.
  \item There exists an invariant dimension-measure on $\cA$ associated to some linear basis.
 \end{enumerate}
\end{thmintro}
By removing the requirement of finite generation, unitality, and having no zero-divisor, we can greatly expand the scope of examples subject to the study of amenability. Of foremost 
interest to us in this paper are two classes of algebras associated to geometric data:
\begin{enumerate}
 \item \emph{Leavitt path algebras} constructed from directed graphs (Definition~\ref{def:Leavitt-path-alg}): These algebras were 
 introduced in \cite{AA} and \cite{AMP} as generalizations of the classical algebras studied by Leavitt in   
 \cite{Leav, LeavDuke}. They also provide natural purely algebraic analogues of the widely studied graph $C^*$-algebras (see e.g. \cite{Raeburn}).
 The class of Leavitt path algebras has interesting connections with various branches of mathematics, such as representation theory, ring theory, group theory,
 and dynamical systems. We refer the reader to  \cite{Abrams15} for a recent survey on this topic.
 \item \emph{Translation algebras} constructed from (locally finite) metric spaces (Definition~\ref{def:translation-algebra}): \\
 These algebras were introduced by Roe as an intermediate step between coarse metric spaces and a class of $C^*$-algebras now known as the (uniform) Roe $C^*$-algebras, as part of his far-reaching work on coarse geometry and the index theory for noncompact manifolds and metric spaces (cf., \cite{Roe03}). Their geometric nature enable them to serve as an important bridge between coarse geometry and the field of operator algebras, as well as a rich source of examples. We will further explore their connections to the theory of $C^*$-algebras in relation to amenability-type properties in \cite{ALLW-2}. 
\end{enumerate}
Typically speaking, these algebras carry zero-divisors, and the translation algebras even have uncountable dimensions.

As corollaries of Theorem~\ref{theorem:aa-characterized-intro}, we observe that \emph{properly infinite} unital algebras are always non-amenable. Recall that a unital algebra $\cA$ is said to be properly infinite if the 
unit is Murray-von~Neumann equivalent to two 
mutually orthogonal idempotents. This condition itself expresses a form of paradoxicality, one that is generally strictly stronger than the notion of
paradoxical decompositions used in Theorem~\ref{theorem:aa-characterized-intro}. This Murray-von Neumann kind of paradoxical decomposition, along with some other forms of non-amenability,
are discussed in \cite[Section 4.5]{Cec01}. 
Indeed, there are division algebras which are non-amenable, and a division algebra cannot be properly infinite (cf. \cite{Elek06}). 
However, proper infiniteness and algebraic non-amenability coincide for the two main classes of examples we study.

\begin{thmintro}[cf., Corollary~\ref{cor:equiv-notamenable-prop-inf} and Theorem~\ref{theorem:main1}]
 Let $\K$ be a field. If $\cA$ is either
 \begin{enumerate}
  \item a unital Leavitt path $\K$-algebra of a finite graph, or
  \item a translation $\K$-algebra (associated to a locally finite extended metric space),
 \end{enumerate}
 then $\cA$ is algebraically amenable if and only if it is not properly infinite. 
\end{thmintro}
In fact, in both cases, we pinpoint the necessary and sufficient properties of the underlying geometric data that give rise to the algebraic amenability of these algebras (cf., Theorem~\ref{thmintro:translation-alg} and Theorem~\ref{thm:amenLPAs}).

One novel aspect of our treatment is the careful distinction, in both the geometric setting and the algebraic setting, between the notion of amenability and the somewhat more restrictive notion of {proper amenability}, which, as described in Definition~\ref{def:alg-amenable-intro}, asks for a F\o lner net that is exhaustive. 
In the group case as well as the case of ordinary metric spaces, these two concepts
coincide (Corollary~\ref{cor:amenisproperamenforordinarymetricspaces}). 
However, subtle differences emerge once we engage \emph{extended} metric spaces, that is, we allow the distance between two points to be infinite. A typical way for this to happen is for an infinite space to admit a finite \emph{coarse connected component} (i.e., a finite cluster of points having finite distances among each other but infinite distances to the rest of the space), as this finite subset would immediately constitute a F\o lner net by itself, which is enough to witness amenability but not enough for proper amenability. In this sense, proper amenability ignores any F\o lner net that comes cheaply from an ``isolated finite substructure''. It turns out such a typical way is, in fact, the only way to separate the two notions in this context (Corollary~\ref{cor:characamenable-nprpamen}). In the algebraic setting, the distinction between the two concepts appears more pronounced, as they possess somewhat different permanence properties (cf., Proposition~\ref{pro:amenability-unitalization}, Example~\ref{eg:alg-amen-left-ideal} and Proposition~\ref{pro:quotient-proper}). Nevertheless, we show that the disagreement between the two notions is always caused by the existence of a finite-dimensional (one-sided) ideal \textemdash~ again a prototypical ``isolated finite substructure'' in the relevant setting. 

\begin{thmintro}[cf., Theorem~\ref{theorem:when-automatically-proper-alg-amen}]
  Let $\cA$ be an infinite dimensional $\K$-algebra over a field $\K$ that is algebraically amenable but not properly algebraically amenable. Then $\cA$ has a finite-dimensional left ideal.
\end{thmintro}
It follows from this theorem that algebraic amenability and proper algebraic amenability also agree for algebras without zero-divisors\footnote{In fact, Elek's original definition in \cite{Elek03} corresponds formally to our definition of proper algebraic amenability, instead of algebraic amenability. For general algebras with possible zero-divisors, we prefer to assign the term ``algebraic amenability'' to the concept without the exhaustion requirement because of its central role in Theorem~\ref{theorem:aa-characterized-intro}.}. The distinction between the two concepts eventually plays a role in the aforementioned generalization of Elek's result in Theorem~\ref{theorem:aa-characterized-intro}, even though the statement of the theorem does not mention proper algebraic amenability. 

Although we only focus on the algebraic and the coarse geometric aspects of amenability in the present article, a major underlying motivation comes from their connections to the F\o lner property in the context of operator algebras. Such connections will be explored in \cite{ALLW-2}, where we will investigate the close relationship between algebraic amenability and the existence of F\o lner nets of projections for operator algebras on a Hilbert space. We remark that F\o lner nets of projections are relevant in single operator theory
\cite{LledoYakubovich13}, operator algebras (see, e.g., \cite{Bedos95,Bedos95,AL14,ALLY14}) as well as in applications to spectral 
approximation problems (see, e.g., \cite{Arveson94,Brown06,Lledo13} and references cited therein). 

We conclude the article with some results connecting the two main objects of study in the paper \textemdash~locally finite (extended) metric spaces and algebras over a field \textemdash~through precisely the construction of the 
{translation algebra} of a locally finite (extended) metric space. With the help of the equivalent characterizations of amenability in both contexts, we obtain the satisfactory result that (proper) amenability of the metric space
is equivalent to (proper) algebraic amenability of the corresponding translation algebra.

\begin{thmintro}[cf., Theorem~\ref{theorem:main1} and Theorem~\ref{theorem:Roe-proper-amenability}]\label{thmintro:translation-alg}
 Let $(X,d)$ be a locally finite extended metric space and let $\K_\mathrm{u}(X)$ be its translation $\K$-algebra of a field $\K$. Then $(X,d)$ is amenable
(respectively, properly amenable) if and only if $\K_\mathrm{u}(X)$ is algebraically amenable (respectively, properly algebraically amenable). 
\end{thmintro}

In the case where the field $\K$ is the complex numbers $\C$, suitable completions of the translation algebras, the so-called \emph{uniform Roe $C^*$-algebras}, will be considered in \cite{ALLW-2}, where further equivalences involving the F\o lner property of these $C^*$-algebras will be established. 

\bigskip

\paragraph{\bf Contents:}
The paper is organized as follows.
In Section~\ref{sec:amenability-metric}, 
we begin by addressing the notion of amenability for locally finite extended metric spaces. 
We will recall in this context the relation to paradoxical decompositions and existence of invariant means
in Theorem~\ref{theorem:amenable-metric}. Finally, we will completely clarify the relation between amenability and 
proper amenability for extended metric spaces in Subsection~\ref{subsec:generalized-metric}.

In Section~\ref{sec:alg-amenable3}, we analyze amenability issues in the context of algebras over a field $\mathbbm K$, and give a complete analysis of the difference between algebraic amenability and proper algebraic amenability
(see Proposition~\ref{pro:amenability-unitalization} and Theorem~\ref{theorem:when-automatically-proper-alg-amen}). If the $\K$-algebra has no zero-divisor, then algebraic amenability and proper algebraic amenability coincide (see Corollary~\ref{cor:non-zero}).

Then we proceed in Section~\ref{sec:dicho-alg-amenable} to develop the relation between algebraic amenability,
paradoxical decompositions and existence of dimension measures on the lattice of subspaces for general $\mathbbm K$-algebras (i.e., not necessarily countably dimensional). This extends previous results by 
Elek in \cite{Elek03} in the context of countably dimensional algebras without zero-divisors.  
In this general setting, and due to the fact that the lattice of subspaces of an algebra is not 
distributive, the notion of additivity and invariance of dimension measures are captured by inequalities instead of equalities
(see Definition~\ref{def:dimension-measure} for details).
Finally, we give examples of how to produce algebras that are not algebraically amenable
using the dimension measure.

In the last two sections, we apply our general theory to two vast classes of examples: the Leavitt path algebras and the translation algebras. In Section~\ref{sec:Leavitt}, we prove that algebraic non-amenability and proper infiniteness coincide for the class of all unital Leavitt path algebras (see Theorem~\ref{thm:amenLPAs}). Using the construction of path algebras, we also give simple examples where left and right algebraic
amenability differ from each other. In Section~\ref{sec:alg-amenable2}, we prove the same result for the class of translation algebras associated to locally finite extended metric spaces. In fact, we also establish equivalences between the algebraic amenability of the translation algebra and the amenability of the underlying metric space (see Theorem~\ref{theorem:main1}), and the
analogous equivalence for proper amenability (see Theorem~\ref{theorem:Roe-proper-amenability}).

\bigskip

\paragraph{\textbf{Notations:}} Given sets $X_1,X_2$ we write their cardinality by $|X_i|$, $i=1,2$ and their 
disjoint union by $X_1\sqcup X_2$. We put $\mathbbm{N}_0=\{0,1,2,\ldots\}=\N\sqcup\{0\}$.

\section{Amenable metric spaces}\label{sec:amenability-metric}

In this section we will study locally finite metric spaces from a large scale geometric 
point of view. There are many interesting examples, of which the most prominent is the case of a finitely
generated discrete group endowed with the word length metric. 
More generally, one can always equip any (countable) discrete group with a right- (or left-)invariant proper metric
and obtain a metric space. The
dependence on the right-invariant proper metric is a rather mild one, if one is
only interested in the ``large-scale'' behavior of the metric space. More
precisely, different right-invariant proper metrics on the same group induce
metric spaces that are \emph{coarsely equivalent}, see, e.g., Section~1.4 in
\cite{Nowak-Yu-12}. Many important properties of groups are ``large-scale'' in nature. Examples include amenability, exactness, 
Gromov hyperbolicity, etc. In this section, we will focus on the first property in this list. 
Amenability has been well studied in coarse geometry
(see, e.g., \cite{Nowak-Yu-12} or \cite[Section~5.5]{bBrown08}), 
so we will only emphasize the aspects which are important for establishing parallelism with the algebraic amenability for 
$\mathbb{K}$-algebras that we are going to investigate in the next sections. 
For the sake of simplicity, we will focus on \emph{locally finite} metric spaces, i.e., those where any bounded set has finite cardinality.\footnote{Recall 
that a metric space is locally finite if and only if it is discrete and proper, the latter meaning that any closed ball
is compact (see, e.g., \cite[Section~5.5]{bBrown08}). We avoid this terminology because we use the term ``proper'' in a different sense in this
article.} 

We start by recalling the definition of amenability for locally finite metric spaces.
Our initial approach will make use of F\o lner sets. Let $(X,d)$ be a metric space and $A$ be a subset of $X$. 
For any $R>0$ define the following natural boundaries of $A$:
\begin{itemize}
 \item {\em $R$-boundary:} $\partial_R A:=\{x\in X: d(x,A)\leq R\ \text{and}\ d(x,X\setminus A)\leq R \}$;
 \item {\em outer $R$-boundary:} $\partial^+_R A := \{x\in X\setminus A:d(x,A)\leq R\}$;
 \item {\em inner $R$-boundary:} $\partial^-_R A := \{x\in A:d(x,X\setminus A)\leq R\}$. 
\end{itemize}
It is clear from the preceding definitions that $\partial_R A=\partial^+_R A\sqcup\partial^-_R A$. 
Next we introduce the notion of amenability of metric spaces due to Block and Weinberger 
(cf., \cite[Section~3]{Block-Weinberger-92}).

\begin{definition}\label{def:metric-amenability}
Let $(X,d)$ be a locally finite metric space. 
\begin{itemize}
 \item[(i)] Let $R>0$ and $\varepsilon\geq 0$. A finite non-empty set $F\subset X$ is called an $(R,\varepsilon)$-\emph{F\o lner set}
  if it satisfies
  \begin{equation*}
  \frac{|\partial_R F|}{|F|}\leq \varepsilon\;.
 \end{equation*}
We denote by $\mathrm{\mbox{F\o l}}(R, \varepsilon)$ the collection of $(R,\varepsilon)$-F\o lner sets.
 \item[(ii)] The metric space $(X,d)$ is called \emph{amenable} if for every $R>0$ and $\varepsilon > 0$ there exists
  $F\in \mathrm{\mbox{F\o l}}(R, \varepsilon)$.
 \item[(iii)] The metric space $(X,d)$ is called \emph{properly amenable} if for every $R>0$, $\varepsilon>0$ 
  and finite subset $A\subset X$ there exists
  a $F\in \mathrm{\mbox{F\o l}}(R, \varepsilon)$ with $A\subset F$.
\end{itemize}
\end{definition}

\begin{remark}\label{rmk:charac-metric-amenability}
 Since with regard to the relation of set containment, $\mathrm{\mbox{F\o l}}(R, \varepsilon)$ is monotonically decreasing with respect to $R$ 
 and monotonically increasing with respect to $\varepsilon$, we may also employ nets to simplify the quantifier-laden ``local'' condition used in the above definition:
 \begin{itemize}
  \item[(i)] Amenability of $(X,d)$ is equivalent to the existence of a net $\{F_i\}_{i\in I}$ of finite non-empty subsets such that 
  \[
   \lim_{i} \frac{|\partial_R F_i|}{|F_i|} = 0 \;,  
   \quad\mathrm{for~all}\quad R > 0 \;.
  \]
  \item[(ii)] Proper amenability of $(X,d)$ requires, in addition, that this net $\{F_i\}_{i\in I}$ satisfies 
  $ X = \liminf_{i} F_i $, where $ \liminf_{i} F_i := \bigcup_{j \in I} \bigcap_{i \geq j} F_i $.
 \end{itemize}
\end{remark}

\begin{example}
 For a finitely generated discrete group $\Gamma$ equipped with the word length metric
both notions are equivalent to F\o lner's condition for the group (see e.g.,
\cite[Proposition~3.1.7]{Nowak-Yu-12}).
\end{example}

\begin{remark}
 With the convention that for any $x\in X$, $d(x,\emptyset)=\infty$, it is immediate that any finite set is properly amenable. Using the notation
\[
 N_R^+ A:= \{x\in X:d(x,A)\leq R\}
 \quad\mathrm{and}\quad
 N_R^- A:= \{x\in X:d(x,X\setminus A) > R\}\;,
\]
we get the relations $\partial_R (N_{R}^+ A) \subset \partial^+_{2R} A $ and
$ \partial_R (N_{R}^-A) \subset \partial^-_{2R} A$. This shows
that for both of the concepts of amenability in Definition~\ref{def:metric-amenability}, 
the use of the $R$-boundary may be replaced by either the outer or the inner $R$-boundary.
\end{remark}

\begin{remark}
 From a coarse geometric point of view, the notion of (proper) amenability as defined above is better behaved when we restrict to metric spaces that are 
 \emph{uniformly locally finite} (some authors call them metric spaces with \emph{bounded geometry}) in the sense that for any $R>0$, 
 there is a uniform finite upper bound on the cardinalities of \emph{all} closed balls with radius $R$, i.e.,
 \begin{equation}
  \label{eq:bdd-geo}   \sup_{x\in X}|B_R(x)|<\infty \;, 
 \end{equation} 
 where $B_R(x):=\{y\in X:d(x,y)\leq R\}$ denotes the closed ball centered at $x$ with radius $R$. The reason is that, for this class of metric spaces, 
 amenability is preserved under coarse equivalence, and this gives us a natural way to generalize the definition to non-discrete metric spaces 
 (satisfying a suitable notion of bounded geometry), cf., \cite[Proposition 3.D.32 and Definiton 3.D.33]{Cornulier-Harpe-14} or 
 \cite[Corollary~2.2 and Theorem~3.1]{Block-Weinberger-92}. This also holds true for proper amenability, with essentially the same argument (perhaps more easily seen with the aid of Lemma~\ref{lem:proper-cardinality} below).  However, for the results we are going to present, we generally do not require our metric space to be uniformly locally finite. 
\end{remark}

The following lemma shows that the definition of proper amenability can be already characterized in terms of the cardinality of the F\o lner sets.

\begin{lemma}\label{lem:proper-cardinality}
 Let $(X,d)$ be an infinite locally finite metric space. 
 Then $X$ is properly amenable if and only if for every $R>0$, $\varepsilon>0$ and  $N\in\mathbb{N}$
 there exists an $F\in\mathrm{\mbox{F\o l}}(R,\varepsilon) $ such that 
 $|F| \geq N$.
\end{lemma}
\begin{proof}
The ``{\em only if}'' part is clear: for any $N\in\N$ just take a finite $A\subset X$
with $|A|=N$. To show the reverse implication let $R>0$, $\varepsilon>0$ and a finite
$A\subset X$ be given. By assumption there is a finite $F\subset X$ such 
that 
\[
 |F|\geq \frac{2 | \partial_R A|}{\varepsilon}
 \quad\mathrm{and}\quad
 \frac{|\partial_R F|}{|F|}\leq \frac{\varepsilon}{2}\;.
\]
Putting $\widetilde{F}:=F\cup A$ (which contains $A$) we have
\[
 \frac{|\partial_R \widetilde{F}|}{|\widetilde{F}|}
     \leq\frac{|\partial_R F|}{|F|} + \frac{|\partial_R A|}{|F|} 
     \leq \frac{\varepsilon}{2}+\frac{\varepsilon}{2}= \varepsilon 
\]
and the proof is concluded.
\end{proof}

As in the group case, the notion of amenability for metric spaces comes with an important dichotomy in relation to
paradoxical decompositions. To formulate it, we first need to
introduce an important tool in the study of coarse geometry.

\begin{definition}\label{def:part-transl}
Let $(X,d)$ be a locally finite metric space. A {\em partial translation on} $X$ is a triple $(A,B,t)$ consisting of two subsets $A$ and $B$ of $X$ 
together with a bijection $t\colon A\rightarrow B$ such that the graph of $t$ given by
\begin{align*}
\text{graph}(t):=\{(x,t(x))\in X\times X:x\in A\}
\end{align*}
is controlled, i.e., $\sup_{x\in A}d(x,t(x))<\infty$. We denote the corresponding
domain and range of $t$ by $\mathrm{dom}(t):=A$ and $\mathrm{ran}(t):=B$. 

The set of all partial translations of $X$ is denoted as $\mathrm{PT}(X)$.
\end{definition}

Note that $\mathrm{PT}(X)$ forms a subsemigroup of the inverse semigroup of partially defined bijective maps $X$ (see, e.g., \cite{Exel98}). More explicitly, the composition of any 
two partial translations $t, t' \in \mathrm{PT}(X)$, denoted by $t \circ t'$, is defined to be the partial translation satisfying 
\[
 \mathrm{dom}( t \circ t') = \left\{ x \in \mathrm{dom}(t') \;|\; t'(x) \in \mathrm{dom}(t) \right\}
\]
and $( t \circ t') (x) = t(t'(x))$ for any $x \in \mathrm{dom}( t \circ t')$. Note that the graph of $t \circ t'$ is also 
controlled since
\[
 \sup_{x\in \mathrm{dom}( t \circ t')}d\left(x, ( t \circ t')(x)\right) \leq \sup_{x\in \mathrm{dom}(t')}d(x,t'(x)) + \sup_{x\in \mathrm{dom}(t)}d(x,t(x)) < \infty \; .
\]

\begin{definition}\label{def:invariant-mean}
A mean $\mu$ on a locally finite metric space $(X,d)$ is a normalized, finitely additive map on the set of all subsets of $X$, $\mu\colon \mathcal{P}(X)\to [0,1]$.
The measure $\mu$ is called {\em invariant under partial translations} if $\mu(A)=\mu(B)$ for all partial translations $(A,B,t)$.
\end{definition}

\begin{definition}\label{def:paradoxical-decomposition}
Let $(X,d)$ be a locally finite metric space. A {\em paradoxical decomposition of} $X$ is a (disjoint) partition 
$X= X_+ \sqcup X_-$ such that there exist two partial translations $t_i:X\rightarrow X_i$ for $i \in \{+, - \}$.
\end{definition}

\begin{remark}\label{rem:Schroeder-Berstein}
 Applying a Bernstein-Schr\"{o}der-type argument, one may slightly weaken the condition of having a paradoxical decomposition: it suffices to assume that 
 there are two disjoint (non-empty) subsets $X'_+, X'_- \subset X$ such that there exist partial translations $t'_i \colon X \to X'_i$ for $i \in \{+,-\}$. 
 Here we do not require their union to be $X$, in contrast with Definition~\ref{def:paradoxical-decomposition}. Indeed, assume we can find $(X'_+, t'_+, X'_-, t'_-)$ as above. 
 We may then write $X = X'_+ \sqcup X'_- \sqcup \widetilde{X}$. Now we define $\displaystyle \widehat{X} = \bigcup_{k=0}^\infty (t'_+)^k (\widetilde{X})$, where 
 $(t'_+)^0$ is viewed as the identity map. This is a disjoint union because $\widetilde{X}$ is disjoint from the image of $t'_+$. 
 Note also that $t'_+$ maps $\widehat{X}$ and $X \setminus \widehat{X}$ into themselves, respectively, and $\widehat{X} = \widetilde{X} \sqcup t'_+( \widehat{X} ) $. 
 By the injectivity of $t'_+$, we have $t'_+(X \setminus \widehat{X}) = X'_+ \setminus t'_+ (\widehat{X} ) = X'_+ \setminus \widehat{X}$. 
 This allows us to construct a paradoxical decomposition $(X_+, t_+, X_2, t_2)$ in the sense of Definition~\ref{def:paradoxical-decomposition} by setting $X_+ = X'_+ \sqcup \widetilde{X}$ (which is equal to $(X'_+ \setminus \widehat{X}) \sqcup \widehat{X}$), 
 $X_2 = X'_-$, $t_+ = \left( t'_+ | _{X \setminus \widehat{X}} \right) \sqcup \mathrm{Id}_{\widehat{X}}$ and $t_2 = t'_-$. 
\end{remark}

The following result gives some standard characterizations of amenable metric spaces that will be 
used later (see, e.g., \cite[Theorems~25 and 32]{Silberstein-Grigorchuk-Harpe-99}; we give an alternative proof
of the implication (\ref{item:X-no-parox})$\Rightarrow$(\ref{item:X-amen})  in the more 
general context of extended metric spaces; see in Theorem~\ref{theorem:amenable-extended-metric}).

\begin{theorem}\label{theorem:amenable-metric}
Let $(X,d)$ be a locally finite metric space. Then the following conditions are equivalent:
\begin{enumerate}
 \item \label{item:X-amen} $(X,d)$ is amenable.
 \item \label{item:X-no-parox} $X$ admits no paradoxical decomposition.
 \item \label{item:inv-meas} There exists a mean $\mu$ on $X$ which is invariant under partial translations.
\end{enumerate}
\end{theorem}
\begin{remark}
Deuber, Simonovits and S{\'o}s in \cite{DSS} considered the exponential growth rate\footnote{It is also called \emph{doubling condition} in the survey of Elek 
and S{\'o}s \cite{ES} and in \cite{Silberstein-Grigorchuk-Harpe-99}.} on locally finite metric spaces and they showed that this growth condition characterizes paradoxicality completely. 
It can be regarded as a Tarski-alternative-type theorem for locally finite metric spaces and it also served as an inspiration for the proof of the Tarski alternative (see \cite[Theorem~32]{Silberstein-Grigorchuk-Harpe-99}).
\end{remark}

It is interesting to note that the notions of paradoxicality and invariant means have been recently introduced and studied for arbitrary Boolean inverse monoids
in \cite{KLLR16}.

\subsection{Amenability vs.~ proper amenability for extended metric spaces}\label{subsec:generalized-metric}

In many ways, the amenability for metric spaces generalizes the
corresponding notion for groups, with certain properties paralleling those of the
latter. However, caution should be taken when one tries to understand amenability for metric spaces
from its similarity with groups. For example, amenability for metric spaces does not
pass to subsets in general. As an example consider the free group $\mathbb{F}_n$, $n\geq 2$,
with a ray attached to it. In this 
sense there is also a parallelism with the notion of F\o lner sequence in the context of
operator algebras as considered in \cite[Section~4]{AL14}.

In this subsection we complete the analysis of amenability in relation to proper amenability in the metric space context.
We shall see that going beyond ordinary metric space (meaning the distance of any two points is finite) helps us better understand some aspects of
amenability. For this we consider {\em extended} metric spaces $(X,d)$ as coarse spaces, i.e., spaces where the metric is allowed to take the 
value $\infty$,
\[ 
d\colon X\times X \to [0, \infty] \;.
\]
For now let us stay assured that the additional complexity brought about by such a generalization is rather mild. Indeed, observe that
the property that two points have finite distance defines an equivalence relation, which decomposes $X$ uniquely into a disjoint union of
equivalence classes $X= \bigsqcup_{i \in I} X_i$, such that each $(X_i, d|_{X_i \times X_i})$ is an ordinary metric space, while 
$d(X_i, X_j) = \infty$ for any different $i,j \in I$. Each $X_i$ is called a \emph{coarse connected component} of $X$. 
Note that if $(X,d)$ is a locally finite {\em extended} metric space, then each component $X_i$ 
is countable although the total space $X$ need not be countable in general. 
As in the usual metric space situation we also have here that if $X$ is finite, then it is properly amenable by taking $F=X$.
As we will show later (Corollary~\ref{cor:amenisproperamenforordinarymetricspaces} and Corollary~\ref{cor:characamenable-nprpamen}), it turns out 
that the notions of amenability and proper amenability are equivalent if the 
extended metric space contains only one coarse connected component (i.e., in the metric space case), but not in general. 

\begin{remark}\label{rem:notions-for-extended-metric}
 Definition~\ref{def:metric-amenability}, Definition~\ref{def:part-transl}, Definition~\ref{def:invariant-mean} and Definition~\ref{def:paradoxical-decomposition} generalize directly to extended metric spaces. So does the Bernstein-Schr\"{o}der-type argument in Remark~\ref{rem:Schroeder-Berstein}.
\end{remark}

\begin{remark}\label{coarse-boundaries}
We will justify here that the characterization of proper amenability in terms of the cardinality
of the F\o lner sets given in Lemma~\ref{lem:proper-cardinality} is still true in the extended 
metric space context. 
Note first that if $F\subset X=\bigsqcup_{i \in I} X_i$ is a finite set (and denoting by $F_i$
the corresponding subset in each coarse connected component $X_i$) we have that
$d(x,F)=\min\{d(x,F_i) : i\in I\}$. Therefore, the $R$-boundary of $F$ decomposes as 
$R$-boundaries in each coarse connected components:
\[
 \partial_R(F)=\mathop{\bigsqcup}_{i \in I} \partial_R (F_i)\;.
\]
(Note also that if $F_i=\emptyset$, then $\partial_R (F_i)=\emptyset$).
Therefore we can reason in each coarse connected component as in
the proof of Lemma~\ref{lem:proper-cardinality}.
\end{remark}

\begin{proposition}
\label{prop:components-and-amenability}
Let $(X,d)$ be a locally finite extended metric space. 
Then $X$ is amenable if at least one of its coarse connected
components is amenable. The converse is true in the case where there are only a finite number of coarse connected components.  
\end{proposition}
\begin{proof}
The first statement is trivial. For the second, assume that $X = \bigsqcup _{i=1}^N X_i$ is a union of finitely many coarse connected components $X_i$,
and that all the coarse connected components are non-amenable. We have to show that $X$ is non-amenable. 
Since all coarse connected components $X_i$ are non-amenable, it follows from 
Theorem~\ref{theorem:amenable-metric} that each component $X_i$ has a paradoxical decomposition. Since there is only 
a finite number of components,
these paradoxical decompositions  can
be assembled to a paradoxical decomposition of $X$, hence $X$ is non-amenable, as desired.
\end{proof}

The second part of Proposition~\ref{prop:components-and-amenability} cannot be generalized to extended metric spaces with an infinite number of 
coarse connected components, as the following example shows.

\begin{example}
 \label{exam:infinite-components-amen}
We construct a locally finite extended metric space $(X,d)$, with an infinite number of coarse connected components, such that neither 
of the connected components of $X$ is amenable, but $X$ is properly amenable. Let $Y$ be the Cayley graph of the 
free non-Abelian group $\mathbb F _2$ of rank two.
For each $n\in \N$, let $Y_n$ be the graph obtained by attaching $n$ new vertices $v_1,\ldots , v_n$ and 
$n$ new edges $e_1,\ldots , e_n$ to $Y$, in such a way that $e_i$ connects $v_i$ with $v_{i+1}$ 
for $i=1,\ldots ,n-1$, and $e_n$ connects $v_n$ with $e$, being $e$ the neutral element of $\mathbb F_2$ (seen as a vertex of $Y$).
Note that $Y_n$ is the graph obtained by attaching a trunk of length $n$ to $Y$. Let $X_n$ be the metric space 
associated to the connected graph $Y_n$, and observe that all the metric spaces $X_n$ are non-amenable. 
Let $X$ be the extended metric space having the metric spaces $X_n$ as coarse connected components. Then clearly $X$ is properly amenable, because
we can use the long trunks to localize the F\o lner sets of $X$ of arbitrary large cardinality.
\end{example}

We also remark that Theorem~\ref{theorem:amenable-metric} given in \cite{Silberstein-Grigorchuk-Harpe-99} stays true in the case of extended metric space. 

\begin{theorem}\label{theorem:amenable-extended-metric}
 Let $(X,d)$ be a locally finite extended metric space. Then the following conditions are equivalent:
\begin{enumerate}
 \item \label{item:X-ext-amen} $(X,d)$ is amenable.
 \item \label{item:X-ext-no-paradox} $X$ admits no paradoxical decomposition.
 \item \label{item:ext-inv-meas} There exists a mean $\mu$ on $X$ which is invariant under partial translations.
\end{enumerate} 
\end{theorem}
\begin{proof}
 The proofs of the implications (\ref{item:X-ext-amen})$\Rightarrow$(\ref{item:ext-inv-meas}) and 
 (\ref{item:ext-inv-meas})$\Rightarrow$(\ref{item:X-ext-no-paradox}) are standard and apply equally well to the extended metric space situation (see, e.g., \cite[\S 26 and part III]{Silberstein-Grigorchuk-Harpe-99}).
 
 The implication (\ref{item:X-ext-no-paradox})$\Rightarrow$(\ref{item:X-ext-amen}) is more interesting. Hereby we present a direct proof for the sake of completeness, adapting ideas from Kerr and Li in \cite[Theorem 3.4, (vi) $\Rightarrow$ (v)]{KL15} to the setting of extended metric spaces (see also \cite{KL15}). 
 This proof should also serve as a motivation for the proof of Proposition~\ref{prop:non-amenable-implies-paradoxic} in the context of algebraic amenability. 

 We suppose that $(X,d)$ is not amenable and would like to show that $X$ has a paradoxical decomposition. By Remark~\ref{rem:Schroeder-Berstein}, it suffices to show that there are two disjoint subsets $X'_+, X'_- \subset X$ such that there exist partial translations $t'_i \colon X \to X'_i$ for $i \in \{+,-\}$. By the negation of Definition~\ref{def:metric-amenability}, there is $\varepsilon_0 \in (0, 1)$ and $R_0>0$
such that, for any finite non-empty set $F\subset X$, one has the following estimate for the outer $R$-boundary: $|\partial^+_{R_0} F|>\varepsilon_0 |F|$
and, hence, $|N^+_{R_0} F|>(1+\varepsilon_0) |F|$. Since, for any finite set $F\subset X$, we also have
\[
 N^+_{2R_0}(F)\ge  N^+_{R_0}\left( N^+_{R_0} F  \right)\geq (1+\varepsilon_0)| N^+_{R_0} F  | \geq (1+\varepsilon_0)^2 | F  | \; ,
\]
we can choose a radius $R_d:=nR_0$ for some $n\geq \log_{1+\varepsilon_0}(2) +1$ satisfying the following local doubling condition:
for any finite non-empty set $F\subset X$, we have 
\[  
|N^+_{R_d} F  | > 2 \;|F|\;.
\]

In the next step of the proof we will essentially use Zorn's lemma to produce a paradoxical decomposition (a ``global doubling'') of $X$. Consider the set
$\Omega$ of set-valued maps $\omega\colon X\times \{ {\scriptstyle +,-}\}\to \cP(X)$ (the power set of $X$) such that for any 
$y=(x,j)\in X\times\{{\scriptstyle +,-}\}$ we have $\omega(y)\in\cP\left( B_{R_d}(x)\right)$ and for any finite set 
$K\subset X\times\{{\scriptstyle +,-} \}$ we have 
\[
 \left| {\bigcup}_{y\in K} \omega(y) \right| \geq |K|\;.
\]
Note that the set $\Omega$ is not empty since the set-valued map given by $\omega(y):=B_{R_d}(x)$ for any $y=(x,j)\in X\times\{{\scriptstyle +,-} \}$
is an element of $\Omega$. In fact, we only need to verify the preceding inequality: 
for any finite set $K \subset X\times\{{\scriptstyle +,-} \}$, we write $K = K_+\times\{{\scriptstyle +}\}\sqcup K_-\times\{{\scriptstyle -}\} $ and calculate that
\[
 \left| {\bigcup}_{y\in K} \omega(y) \right| = \left| N_{R_d}^+(K_+\cup K_-)  \right|\geq 2|K_+\cup K_-|\geq |K_+|+|K_-|=|K|\;.
\]
The set $\Omega$ may also be partially ordered in the following natural way
\[
 \omega\leq \omega' \quad\mathrm{if}\quad \omega(y)\subset \omega'(y)\quad\mathrm{for~any~} y\in X\times\{{\scriptstyle +,-}\}\;.
\]
Since any descending chain has a non-empty lower bound given by pointwise intersection we obtain by Zorn's lemma a minimal element 
$\omega_m\in\Omega$. Note that, by the definition of $\Omega$, we already have $| \omega_m(y) |\geq 1$ for any $y\in X\times\{{\scriptstyle +,-}\}$.

We claim that $|\omega_m(y)|= 1$ for any $y\in X\times\{{\scriptstyle +,-}\}$. Suppose this is not the case. Then there is $y_0\in X\times\{{\scriptstyle +,-}\}$ such that $\omega_m(y_0)$ has two distinct elements $x_+,x_-$. By the minimality of $\omega_m$, there exist, for 
$l\in\{{\scriptstyle +,-}\}$, finite sets $K_l\subset  X\times\{{\scriptstyle +,-}\}$ not containing $y_0$ and such that
\[
 \left| \left(\omega_m(y_0)\setminus\{x_l\} \right) \cup \left({\bigcup}_{y\in K_l} \omega_m (y)\right) \right|  \leq  |K_l| \;.
\]
(Note that, otherwise, one could remove $x_l$ from $\omega_m(y_0)$ to specify a new element in $\Omega$ strictly smaller than 
$\omega_m$.) Define, for $l\in\{{\scriptstyle +,-}\}$, the set
\[
 Z_l:=  \left(\omega_m(y_0)\setminus\{x_l\} \right) \cup \left({\bigcup}_{y\in K_l} \omega_m (y)\right)\;.
\]
Using the identity $(\omega_m(y_0)\setminus\{x_+\})\cup(\omega_m(y_0)\setminus\{x_-\})=\omega_m(y_0)$ as well as the preceding inequality, we obtain the following contradiction:
\begin{eqnarray*}
  |K_+|+|K_-| &\geq& |Z_+|+|Z_-| \;=\;|Z_+\cup Z_-|+|Z_+\cap Z_-| \\[2mm]
      &\geq& \left|\omega_m(y_0)\cup \left(\mathop{\bigcup}_{y\in (K_+\cup K_-)} \omega_m(y)\right) \right| 
                    + \left|\mathop{\bigcup}_{ y\in (K_+\cap K_-)} \omega _m(y) \right| \\[2mm]
      &\geq& 1+|K_+\cup K_-|+|K_+\cap K_-|= 1+|K_+|+|K_-|\;.
 \end{eqnarray*}
Therefore $|\omega_m(y)|= 1$ for any $y\in X\times\{{\scriptstyle +,-}\}$. 

To finish the proof, we define, for any $l\in\{{\scriptstyle +,-}\}$, the map $t_l\colon X\to X_l$ which assigns to each $x\in X$ the unique element in $\omega_m(x,l)$. Note that it follows now from the definition of $\Omega$ that $\omega_m(y) \cap \omega_m(y') = \varnothing$ if $y\not= y'$. Consequently, both $t_+$ and $t_-$ are injective and they have disjoint images, which we denote by $X_+$ and $X_-$, respectively. Since by definition $\omega_m(x,l)\subset B_{R_d}(x)$ we have
\[
 \sup \{d(x,t_l(x)) : x\in X\}\leq R_d\;,
\]
hence the maps $t_+$, $t_-$ are controlled and the quadruple $(X_+, t_+, X_-, t_-)$ satisfies the condition in Remark~\ref{rem:Schroeder-Berstein}
and, hence, a paradoxical decomposition can be obtained from them.
\end{proof}

The next proposition is the key to our results on the relationship between amenability and proper amenability for extended metric spaces.

\begin{proposition}\label{prop:prepare-for-amennotpamen}
Let $(X,d)$ be a non-empty locally finite extended metric space, and assume that all the coarse connected components of 
$X$ are infinite. Then $X$ is amenable if and only if $X$ is properly amenable.    
\end{proposition}

\begin{proof}
Suppose that $X= \bigsqcup_{i\in I} X_i$ is amenable, where $X_i$ are the coarse connected components of $X$.
By Remark~\ref{coarse-boundaries}, it is enough to check that for $R>0$ and $\varepsilon>0$ the sets in $\mathrm{\mbox{F\o l}}(R, \varepsilon)$
have unbounded cardinality. Suppose this is not the case, i.e., there is $R_0>0$, $\varepsilon_0$ with $1>\varepsilon_0>0$ and 
$N_0\in\N$ such that $\mathrm{\mbox{F\o l}}(R_0, \varepsilon_0)$ has an element $F_0$ of maximal cardinality $N_0$.
Write $F_0= \bigsqcup_{i\in I_0} F_{0,i}$, where $F_{0,i}$, $i\in I_0$, are the (non-empty) coarse connected components of $F_0$, so that $F_{0,i}= F_0\cap X_i\ne \emptyset $ for 
$i\in I_0$, and $I_0$ is a finite subset of $I$.   
Set 
$$R_1:= \mathrm{max}_{i\in I_0} \{ \mathrm{diam}(F_{0,i}) + \mathrm{dist}(F_{0,i}, X_i\setminus F_{0,i}) \} , $$
where $\mathrm{diam}(F_{0,i})=\max{\{d(x,y) : x,y\in F_{0,i}\}}$ is the diameter of $F_{0,i}$.
Observe that $X_i \setminus F_{0,i}$ is non-empty by our hypothesis that all the coarse connected components of $X$ are infinite.

Choose $R>R_0+R_1$,
and $\varepsilon>0$ such that 
$$\varepsilon < \mathrm{min} \Big{\{} \varepsilon _0, \frac{1}{|F_0|} \Big{\}} .$$ 
 Since $X$ is amenable, there exists 
$F\in \mathrm{\mbox{F\o l}}(R, \varepsilon )$. We claim that $F\not\subset F_0$. Indeed, if $F\subset F_0$, then by the choice of $R_1$ we have
$F_{0,i} \subset \partial_R F_i$ for all $i\in I_0$ such that the coarse connected component $F_i$ of $F$ is non-empty. Let $I_0'$ be the (non-empty) subset of $I_0$
consisting of those $i\in I_0$ such that $F_i\ne \emptyset$. Then we obtain 
$$\frac{|\partial_R F|}{|F|} \ge \frac{\sum_{i\in I_0'}|F_{0,i}|}{\sum_{i\in I_0'}|F_{0,i}|} = 1 > \varepsilon _0 > \varepsilon, $$
hence $F\notin \mathrm{\mbox{F\o l}}(R, \varepsilon )$, proving our claim. 
Write $F=\bigsqcup_{j\in J_0} F_j$, where $J_0$ is finite, and $\{ F_j : j\in J_0 \}$ are the (non-empty) coarse connected components of $F$.
It follows that for some coarse connected component $F_{j_0}$ of $F$, we have $F_{j_0}\not\subset F_{0}$. 

For $k\in I_0\cup J_0$, set
$F_{0,k} = F_0\cap X_k$ and $F_k= F\cap X_k$. (Note that some $F_{0,k}$ or some $F_k$ might be empty.)

We consider next two cases:
\begin{itemize}
 \item[(a)] If $\partial_R(F) \ne \emptyset$, then 
$$\frac{1}{|F|} \le \frac{|\partial_R(F)|}{|F|} \le \varepsilon <\frac{1}{|F_0|}$$
and so, $N_0= |F_0|< |F|$. Hence $F\in \mathrm{\mbox{F\o l}}(R_0, \varepsilon_0)$ with $|F|>N_0$, which is a contradiction
to the maximality of $N_0$.

\item[(b)] If $\partial_R (F) = \emptyset$ we have two possibilities, for each $j\in J_0$:
\begin{itemize}
 \item[(i)] If $F_j\cap F_{0, j}\not=\emptyset$, then  $F_{0,j} \subset F_j $ by using our assumption that $\partial_R(F)=\emptyset$.
 \item[(ii)] $F_j \cap F_{0,j}=\emptyset$.
\end{itemize}
Assume that condition (ii) holds for some $j_0\in J_0$. Then $\widetilde{F}:= F_0\sqcup F_{j_0}$ satisfies 
$$ \frac{|\partial_{R_0} (\widetilde{F})|}{|\widetilde{F}|}
                      \leq  \frac{|\partial_{R_0} (F_0)|+|\partial_{R_0} (F_{j_0})|}{|F_0|+|F_{j_0}|}
                       = \frac{|\partial_{R_0} (F_0)|}{|F_0|+|F_{j_0}|} < \frac{|\partial_{R_0} (F_0)|}{|F_0|}\le \varepsilon_0  ,$$
where the equality follows from the fact that $\partial_R (F) = \emptyset$.
Thus $\widetilde{F}$ is a $(R_0,\varepsilon_0)$-F\o lner set with $| \widetilde{F}| >N_0$ and we have a contradiction.  

If case (i) occurs for all $j\in J_0$, then $J_0\subset I_0$ and $F_{0,j}\subset F_j$ for all $j\in J_0$. Writing $\widetilde{F} = F_0 \cup F$, we have that $|\widetilde{F}| >|F_0|= N_0$, because
$F\not\subset F_0$. Setting $I_0'':= I_0\setminus J_0$, we get, using that $\partial_{R_0}F_j= \emptyset$ for all $j\in J_0$,
\begin{align*}
 \frac{|\partial_{R_0}\widetilde{F}|}{|\widetilde{F}|}  & = \frac{\sum_{j\in J_0} |\partial_{R_0}F_j| + \sum _{i\in I_0''} |\partial_{R_0}F_{0,i}|}{|\widetilde{F}|}\\
 & = \frac{ \sum _{i\in I_0''} |\partial_{R_0}F_{0,i}|}{|\widetilde{F}|} 
  \le  \frac{|\partial_{R_0}F_0|}{|F_0|}\le \varepsilon_0,   
  \end{align*}
so that $\widetilde{F}$ is a $(R_0,\varepsilon_0)$-F\o lner set
of cardinality strictly larger than $N_0$, which is again a contradiction. 
\end{itemize}
In either case we get a contradiction to the maximality of $N_0$ and the proof is concluded.
 \end{proof}

As an immediate consequence of Proposition~\ref{prop:prepare-for-amennotpamen}, we obtain the following result.

\begin{corollary}\label{cor:amenisproperamenforordinarymetricspaces}
 Let $(X,d)$ be a locally finite metric space. Then $(X,d)$ is amenable if
and only if $(X,d)$ is properly amenable.
\end{corollary}

We can now obtain the characterization of the amenable but not properly amenable extended metric spaces. This should be compared to 
Theorem~\ref{theorem:when-automatically-proper-alg-amen} in the algebraic setting.

\begin{corollary}\label{cor:characamenable-nprpamen}
Let $(X,d)$ be a locally finite extended metric space with infinite cardinality. 
Then $X$ is amenable but not properly amenable if and only if $X= Y_1\sqcup Y_2$, where $Y_1$ is a finite non-empty subset of $X$,
$Y_2$ is non-amenable and $d(x,y)= \infty $ for $x\in Y_1$ and $y\in Y_2$. 
\end{corollary}

\begin{proof}
Assume first that $X= Y_1\sqcup Y_2$, where $Y_1$ is a finite non-empty subset of $X$,
$Y_2$ is non-amenable and $d(x,y)= \infty $ for $x\in Y_1$ and $y\in Y_2$. Observe that $Y_1$ is the disjoint 
union of some coarse connected components of $X$, and $Y_2$ is the disjoint union of the rest of the coarse connected components of $X$. 
Clearly $Y_1$ is a finite non-empty subset of $X$ such that $\partial _R(Y_1)=\emptyset$
for all $R>0$. Hence $X$ is amenable. One can easily show that, if $X$ is properly amenable, 
then $Y_2$ is also properly amenable, contradicting our hypothesis. Indeed, given and $R>0$, $\varepsilon >0$ and $N>0$, take 
a subset $F$ of $Y_2$ such that 
$$\frac{|\partial_R(Y_1\sqcup F)|}{|Y_1\sqcup F|} \le \delta ,$$
where $\delta $ satisfies $0 < \delta (1+\delta ) < \varepsilon$, and $|F|\ge \mathrm{max} \{N, \frac{|Y_1|}{\delta}\}$. Then $F$ is a $(R,\epsilon )$-F\o lner subset
of $Y_2$ with $|F|\ge N$, as desired. Hence, $X$ is amenable but not properly amenable. 

Suppose now that $X$ is amenable but not properly amenable. We first show that there are only a finite number of finite components. 
Indeed, if $X_1,X_2,\ldots , $ is an infinite sequence of finite coarse connected components, 
then $\bigsqcup_{i=1}^n X_i$ are F\o lner $(R, 0)$-subsets of unbounded cardinality in $X$, and so 
$X$ is properly amenable by Remark~\ref{coarse-boundaries}, giving a contradiction.
Hence there is only a finite number of finite coarse connected components $X_1,\ldots , X_N$. Let $Y_1= \bigsqcup _{i=1}^N X_i$, and let $Y_2=X\setminus Y_1$.
Then all the coarse connected components of $Y_2$ are infinite. If $Y_2$ is amenable, then it is also properly amenable by  
Proposition ~\ref{prop:prepare-for-amennotpamen}, 
and so $X$ is also properly amenable, contradicting our hypothesis. Hence $Y_2$ is non-amenable. Since $X$ is amenable by hypothesis, we conclude that $Y_1\ne \emptyset$.
This concludes the proof.
\end{proof}

\section{Algebraic amenability}\label{sec:alg-amenable3}

In this section we will analyze from different points of view a version of amenability for $\mathbb{K}$-algebras, 
where $\mathbb{K}$ is a field. 
Our definition will follow existing notions in the literature
(see Section~1.11 in \cite{Gromov99} and \cite{Elek03,Cec-Sam-08}), but we aim to generalize previous definitions and results in a systematical fashion. 
To simplify terminology, we will often not mention $\mathbb{K}$ explicitly. 
For instance, we may call $\mathbb{K}$-algebras just algebras, and $\mathbb{K}$-dimensions just dimensions.

\begin{definition}\label{def:alg-amenable2}
 Let $\cA$ be a $\mathbb{K}$-algebra. 
 \begin{itemize}
  \item[(i)] Let $\mathcal{F}\subset\cA$ be a finite subset and $\varepsilon \geq 0$. Then a nonzero finite-dimensional
   linear subspace $W \subset \cA$ is called a left \emph{$(\mathcal{F}, \varepsilon)$-F\o lner subspace} 
   if it satisfies
  \begin{equation}\label{eq:alg-amen2}
   \frac{\dim(a W +W)}{\dim(W)}\leq 1+\varepsilon\;,  
   \quad\mathrm{for~all}\quad a\in\mathcal{F}\;.
  \end{equation}
  The collection of $(\mathcal{F}, \varepsilon)$-F\o lner subspaces of $\cA$ is denoted by $\mathrm{\mbox{F\o l}}(\cA, \mathcal{F}, \varepsilon)$.
  \item[(ii)] $\cA$ is left \emph{algebraically amenable} if for 
  any $\varepsilon >0$ and any finite set $\mathcal{F}\subset\cA$, there exists a left $(\mathcal{F}, \varepsilon)$-F\o lner subspace.
  \item[(iii)] $\cA$ is \emph{properly} left algebraically amenable if for any $\varepsilon >0$ and any finite set $\mathcal{F}\subset\cA$, 
  there exists a left $(\mathcal{F}, \varepsilon)$-F\o lner subspace $W$ such that $\mathcal{F} \subset W$. 
 \end{itemize}
\end{definition}

We may also define \emph{right} F\o lner subspaces, \emph{right} algebraic amenability and proper \emph{right} algebraic amenability by 
replacing $\cA$ with $\cA^\mathrm{op}$ in the above definitions. Since the two situations are completely symmetric, we will stick with the left
versions of the definitions. For simplicity we are going to drop the term ``left'' for the rest of this section.
Any algebra satisfying $\dim(\cA)<\infty$ is obviously properly algebraically amenable by taking $W=\cA$.

\begin{remark}\label{rmk:charac-alg-amen2}
 There are some slightly different, but equivalent, ways to define (proper) algebraic amenability. For example, since for any 
 $\varepsilon >0$ and any finite set $\mathcal{F}\subset\cA$, an $(\mathcal{F}, \varepsilon)$-F\o lner subspace also satisfies 
 \[
  \frac{\dim( \mathrm{span}( \mathcal{F} W +W))}{\dim(W)}\leq 1+ |\mathcal{F}| \varepsilon\;,
 \]
 we may equivalently define algebraic amenability for $\cA$ as saying that for 
  any $\varepsilon >0$ and any finite set $\mathcal{F}\subset\cA$, there exists a nonzero finite-dimensional linear subspace $W$ such that 
 \[
  \frac{\dim( \mathrm{span}( \mathcal{F} W +W))}{\dim(W)}\leq 1+\varepsilon\;.
 \]
 Since with regard to the relation of set containment, $\mathrm{\mbox{F\o l}}(\cA, \mathcal{F}, \varepsilon)$ is monotonically
 decreasing with respect to $\mathcal{F}$ and monotonically increasing with respect to $\varepsilon$, we may also employ nets to simplify the
 quantifier-laden ``local'' condition used in the above definition:
 \begin{itemize}
  \item[(i)] Algebraic amenability of $\cA$ is equivalent to the existence of a net $\{W_i\}_{i\in I}$ of finite-dimensional linear subspaces such that 
  \[
   \lim_{i} \frac{\dim(a W_i +W_i)}{\dim(W_i)} = 1 \;,  
   \quad\mathrm{for~all}\quad a\in\cA\;.
  \]
  \item[(ii)] Proper algebraic amenability of $\cA$ requires, in addition, that this net $\{W_i\}_{i\in I}$ satisfies 
  $ \cA = \liminf_{i} W_i $, where $ \liminf_{i} W_i := \bigcup_{j \in I} \bigcap_{i \geq j} W_i $.
 \end{itemize}
\end{remark}

\begin{remark}
\begin{itemize}
  \item[(i)] 
 The notion given by Elek in Definition~1.1 of \cite{Elek03} in fact corresponds to \emph{proper} algebraic amenability, as will 
 become evident in the next proposition (see also Definition~3.1 in \cite{Cec-Sam-08}). Nevertheless, since the main results in 
 Elek's paper restrict to the case of algebras with no zero divisors, alebraic amenability and proper algebraic amenability are equivalent
 (see Corollary~\ref{cor:non-zero} below). 
 \item[(ii)]
 In Definition~4.3 of \cite{Bart},  Bartholdi uses the name {\em exhaustively amenable} instead of properly amenable.
\end{itemize}
 \end{remark}

Notice that although the definition works for $\mathbb K$-algebras of arbitrary dimensions, the property of algebraic amenability is in essence a property for countably dimensional algebras, as seen in the next proposition.

\begin{proposition}\label{pro:alg-amenability-countability}
 A $\mathbb K$-algebra $\cA$ is (properly) algebraically amenable if and only if any countable subset in $\cA$ is contained in a countably dimensional $\mathbb K$-subalgebra that is (properly) algebraically amenable.
\end{proposition}

\begin{proof}
 For the forward direction, we assume $\cA$ is (properly) algebraically amenable and let $\mathcal C \subset \cA$ be an arbitrary countable subset. Using the fact that a subalgebra generated by a countable set or a countably dimensional linear subspace is countably dimensional, we define an increasing sequence $\{\cB_i\}_{i=0}^\infty$ of countably dimensional $\mathbb K$-subalgebras in $\cA$ as follows: 
 \begin{itemize}
  \item We let $\cB_0$ be the subalgebra generated by $\mathcal C$.
  \item Suppose $\cB_i$ has been defined. Let $\{e_k\}_{k=1}^\infty$ be a basis of $\cB_i$. By the (proper) algebraic amenability of $\cA$, for each positive integer $k$, we may find a finite dimensional linear subspace $W_k \subset \cA$ that is $(\{e_1, \ldots, e_k\}, \frac{1}{k})$-F\o lner (and contains $\{e_1, \ldots, e_k\}$ in the case of proper algebraic amenability). We define $\cB_{i+1}$ to be the subalgebra generated by the countably dimensional linear subspace $\cB_i + W_1 + W_2 + \ldots$. 
 \end{itemize}
 Now define the countably dimensional subalgebra $\cB = \bigcup_{i=0}^\infty \cB_i$. It is routine to verify that $\cB$ is (properly) algebraically amenable.
 
 Conversely, in order to check (proper) algebraic amenability of $\cA$, we fix $\varepsilon>0$ and an arbitrary finite subset $\cF \subset \cA$. By assumption, $\cF$ is contained in a countably dimensional subalgebra that is (properly) algebraically amenable, which is enough to produce the desired $(\cF, \varepsilon)$-F\o lner subspace. 
\end{proof}

Just as in the case of metric spaces in Section~\ref{sec:amenability-metric}, we are interested in the distinctions and relations between amenability and proper amenability. For example, when $\cA$ is finite dimensional, then the two notions clearly coincide. The general situation bears strong similarity to the case of metric spaces. To begin with, we present a few more ways to characterize proper algebraic amenability (for infinite dimensional algebras). The first half of the following proposition should be considered as the algebraic counterpart
of what we already showed in Lemma~\ref{lem:proper-cardinality} in the context of metric spaces. 

\begin{proposition}\label{pro:proper-alg-amen2}
 Let $\cA$ be an infinite dimensional $\mathbb{K}$-algebra. Then the following conditions are equivalent:
 \begin{itemize}
  \item[(1)] $\cA$ is properly algebraically amenable.
  \item[(2)] For any $\varepsilon >0$, $N \in \N$ and any finite set $\mathcal{F}\subset\cA$ there exists an 
  $(\mathcal{F}, \varepsilon)$-F\o lner subspace $W$ such that
  \[
   \dim(W) \geq  N \;.
  \]
 \end{itemize}
 When $\cA$ is unital, they are also equivalent to
 \begin{itemize}
  \item[(3)] For any $\varepsilon >0$ and any finite set $\mathcal{F}\subset\cA$ there exists an 
  $(\mathcal{F}, \varepsilon)$-F\o lner subspace that contains $ \1_{\cA} $.
 \end{itemize}
\end{proposition}
\begin{proof}
 The implication (1) $\Rightarrow$ (2) is immediate from the definition, since $\cF \subset W$ implies $\dim(W) \geq \dim( \mathrm{span}(\mathcal{F}))$, while the latter may be made arbitrarily large since $\cA$ is infinite dimensional. 
 
 Next we show the converse: (2) $\Rightarrow$ (1). Given any $\varepsilon >0$ and any finite
set $\mathcal{F}\subset\cA$, we may obtain from (2) a finite-dimensional linear subspace $V\subset \cA$ such that $ \dim(V) \geq \frac{4
|\mathcal{F}|}{\varepsilon} $ and 
 \[
  \frac{\dim(a V +V)}{\dim(V)}\leq 1+ \frac{\varepsilon}{2} \;,  
  \quad\mathrm{for~all}\quad a\in\mathcal{F}\;.
 \]
 Define $ W := V + \mathrm{span}(\mathcal{F}) $, a finite-dimensional linear subspace that contains $\mathcal{F}$. Moreover, 
 for all $a\in\mathcal{F}$,
 \[
  \frac{\dim(a W +W)}{\dim(W)} \leq \frac{\dim(a V +V) + \dim( \mathrm{span}(a \mathcal{F}  \cup \mathcal{F}))}{\dim(V)} \leq 1+ \frac{\varepsilon}{2}
+ \frac{\varepsilon}{2} \leq 1+\varepsilon \;.
 \]
 This proves (1) by definition.
 
 Now assume $\cA$ is unital. The implication (1) $\Rightarrow$ (3) is trivial from the definition, while (3) $\Rightarrow$ (2) is also easy in view of
 Remark~\ref{rmk:charac-alg-amen2}, after observing that 
 $ \1_{\cA} \in W $ implies $\dim( \mathrm{span}( \mathcal{F} W +W)) \geq \dim( \mathrm{span}(\mathcal{F}))$. This shows that (3) is equivalent to (1) and (2). 
\end{proof}

A notable difference between algebraic amenability and proper algebraic amenability lies in their behaviors under unitization. Recall that 
for a (possibly unital) $\mathbb{K}$-algebra, the \emph{unitization} of $\cA$, denoted by $\widetilde{\cA}$, is defined to be the unital algebra 
linearly isomorphic to $\cA \oplus \mathbb{K}$, with the product defined by $(a, \lambda) (b, \mu) = (ab + \mu a + \lambda b, \lambda \mu)$ for 
any $(a, \lambda), (b, \mu) \in \cA \oplus \mathbb{K}$. The element $(0,1)$ now serves as the unit $\1_{\widetilde{\cA}}$. 
Observe that when $\cA$ already has a unit, then $\widetilde{\cA} \cong \cA \times \mathbb{K}$ as an algebra. 
 
 \begin{proposition}\label{pro:amenability-unitalization}
  Let $\cA$ be a $\mathbb{K}$-algebra. Then
  \begin{enumerate}
   \item $\widetilde{\cA}$ is algebraically amenable if ${\cA}$ is algebraically amenable.
   \item $\widetilde{\cA}$ is properly algebraically amenable if and only if $\cA$ is properly algebraically amenable.
  \end{enumerate}
 \end{proposition}

 \begin{proof}
  Let $\pi \colon \cA \oplus \mathbb{K} \to \cA$ be the projection onto the first coordinate and $\iota \colon \cA \to \cA \oplus \mathbb{K}$ be the embedding onto 
  $\cA \times \{0\}$. We also assume that $\cA$ is infinite dimensional, as otherwise there is nothing to prove. 
  
  To prove (1), we assume $\cA$ is algebraically amenable. Then for any $\varepsilon > 0$ and any finite subset $\cF \subset \widetilde{\cA}$, we pick an 
  $(\pi(\cF), \varepsilon)$-F\o lner subspace $W$ in $\cA$. Then $\iota(W) \subset \widetilde{\cA}$ is $(\cF, \varepsilon)$-F\o lner because for any 
  $(a,\lambda) \in \cF$, $(a,\lambda) \cdot \iota(W) + \iota(W) = \iota(aW + W)$. Thus $\widetilde{\cA}$ is algebraically amenable.
  
  As for (2), we first observe that the ``if'' part is proved similarly as above, except for that we also use the fact that $\dim(\iota(W)) = \dim(W)$ and apply 
  Proposition~\ref{pro:proper-alg-amen2}.
  
  Conversely, suppose $\widetilde{\cA}$ is properly algebraically amenable. For any $\varepsilon > 0$ and any finite subset $\cF' \subset \cA$, we pick an $(\iota(\cF'), \varepsilon)$-F\o lner subspace $W'$ in $\widetilde{\cA}$ such that $\iota(\cF') \subset W'$. Then for any $a \in \cF'$ and $(b, \mu) \in W'$, 
  we have $\iota(a) \cdot (b, \mu) = \iota(ab + \mu a) \in \iota(ab) + W'$, and thus
  \[
   \pi \big( \iota(a) \cdot W' + W' \big)  = a \cdot \pi(W') + \pi(W') \; .
  \]
  Since $\mathrm{Ker}(\pi) = \mathbb{K} \cdot (0,1)$, we have
  \begin{align*}
   & \frac{\dim_{\mathbb{K}}(a \cdot \pi(W') + \pi(W'))}{\dim_{\mathbb{K}}(\pi(W'))} 
   =  \frac{\dim_{\mathbb{K}}\left(\pi \big( \iota(a) \cdot W' + W' \big)\right)}{\dim_{\mathbb{K}}(\pi(W'))} \\[2mm]
   \in & \ \left\{ \frac{\dim_{\mathbb{K}}\left( \iota(a) \cdot W' + W' \right) }{\dim_{\mathbb{K}}(W')} ,  
          \frac{\dim_{\mathbb{K}}\left( \iota(a) \cdot W' + W' \right) - 1}{\dim_{\mathbb{K}}(W')} , 
          \frac{\dim_{\mathbb{K}}\left( \iota(a) \cdot W' + W' \right) -1 }{\dim_{\mathbb{K}}(W') -1 } \right\} \\[2mm]
   \subset & \ \left[1 , 1 + \frac{\dim_{\mathbb{K}}\left( \iota(a) \cdot W' + W' \right) - \dim_{\mathbb{K}}(W') }{\dim_{\mathbb{K}}(W') -1 } \right] \\[2mm]
   \subset & \ \left[1 , 1 + \varepsilon \left( 1 + \frac{1}{ | \cF' | - 1} \right) \right] \; .
  \end{align*}
  Since without loss of generality, we may assume $|\cF'| \geq 2$, thus $\pi(W')$ is $(\cF', 2 \varepsilon)$-F\o lner and contains $\cF'$. This shows that ${\cA}$ is properly algebraically amenable.
 \end{proof}

The following example exhibits the difference between algebraic amenability and proper algebraic amenability, and also demonstrate that the converse
of (1) in Proposition~\ref{pro:amenability-unitalization} is false (see also Theorem~3.2 in \cite{LledoYakubovich13} for an operator theoretic
counterpart).

\begin{example}\label{eg:alg-amen-left-ideal}
Let $\cA$ be a $\mathbb{K}$-algebra with a non-zero left ideal $I$ of finite $\mathbb K$-dimension. Then $\cA$ is always algebraically
amenable, since $I$ is an $(\cA, \varepsilon=0)$-F\o lner subspace. Therefore an easy way to construct an amenable $\mathbb{K}$-algebra that is not properly amenable is to take a direct sum of a finite dimensional algebra and a non-algebraically-amenable algebra (e.g., the group algebra of a non-amenable group; see Example~\ref{ex:group-algebra}). 
In particular, if $\cA$ is a non-amenable unital algebra, then $\widetilde{\cA} \cong \cA \oplus \mathbb{K}$ is algebraically amenable but not properly algebraically amenable.
Moreover, this is the only way in which a unitization $\widetilde{\cA}$ can be algebraically amenable but not properly algebraically amenable, as we will show in Corollary~\ref{cor:unitizationalgamen-notpaa}.
\end{example}

The next result refers to two-sided ideals. 

\begin{proposition} \label{pro:quotient-proper}
Let $\cA$ be a $\mathbb{K}$-algebra with a non-zero two-sided ideal $I$ of finite $\mathbb K$-dimension. 
Then, $\cA$ is properly algebraically amenable if and only if the quotient algebra $\cA/I$ is.
\end{proposition}
\begin{proof}
Let $\pi\colon  \cA \to \cA / I$ is the natural projection, then for any $\varepsilon >0$ and any finite set
$\mathcal{F}\subset\cA$, $V \mapsto \pi^{-1}(V)$ defines a map 
from $\mathrm{\mbox{F\o l}}(\cA / I, \pi(\mathcal{F}), \varepsilon)$ to $\mathrm{\mbox{F\o l}}(\cA, \mathcal{F}, \varepsilon)$ 
with $\mathrm{dim} (\pi^{-1}(V) ) \geq \mathrm{dim} (V)$.

On the other hand for 
any $\varepsilon >0$ and any finite set $\mathcal{F'}\subset\cA / I$ such that $\mathrm{dim} (\mathrm{span} (\mathcal F ')) > 0$,  
$W \mapsto \pi(W)$ defines a map from $\mathrm{\mbox{F\o l}}^{\pi^{-1}(\mathrm{span}(\mathcal F ')) }(\cA, \pi^{-1}(\mathcal{F'}),
\frac{\varepsilon}{K} )$ to $\mathrm{\mbox{F\o l}}(\cA / I, \mathcal{F'}, \varepsilon)$ with 
$$\mathrm{dim}( \pi(W))  = \mathrm{dim}(W) - \mathrm{dim}(I)\;,$$ where 
$$K= 1 + \frac{\mathrm{dim} (\pi^{-1} (\mathrm{span}(\mathcal F ')))}{\mathrm{dim} (\pi^{-1} (\mathrm{span}(\mathcal F '))) - \mathrm{dim}(I)}\;, $$ 
and $\mathrm{\mbox{F\o l}}^{V} (\cA, \mathcal{F''}, \varepsilon ' )$ is the set of all $W$ in 
$\mathrm{\mbox{F\o l}} (\cA, \mathcal{F''}, \varepsilon' )$ 
such that $V\subseteq W$, for any finite-dimensional subspace $V$ of $\mathcal A$. Indeed, 
for $W$ in $\mathrm{\mbox{F\o l}}^{\pi^{-1}(\mathrm{span}(\mathcal F ')) }(\cA, \pi^{-1}(\mathcal{F'}),
\frac{\varepsilon}{K} )$, we have
\begin{align*}
  & \frac{\mathrm{dim} (\pi (aW+W))}{\mathrm{dim} (\pi(W))}  = \frac{\mathrm{dim} (aW+W) -\mathrm{dim} (I)}{ \textrm{dim} (W) - \textrm{dim}(I)} \\
 & = \frac{\text{dim} (aW+W)}{\textrm{dim}(W)} \cdot \frac{\textrm{dim} (W)}{\textrm{dim} (W) - \textrm{dim} (I)} + \Big( 1 - \frac{\textrm{dim} (W)}{\textrm{dim} (W) - \textrm{dim} (I)} \Big) \\
 & = \Big( \frac{\textrm{dim} (aW+W)}{\textrm{dim} (W)} -1  \Big) \Big( \frac{\textrm{dim} (W)}{\textrm{dim} (W)- \textrm{dim} (I)} \Big)  +1 \\
&  \le \frac{\varepsilon}{K} \Big( \frac{\textrm{dim}(W)}{\textrm{dim} (W)- \textrm{dim} (I)} \Big) +1 ,  
 \end{align*}
and it is easily seen that $$\frac{1}{K}\Big( \frac{\textrm{dim}(W)}{\textrm{dim} (W)- \textrm{dim} (I)} \Big)= 
\Big( \frac{\mathrm{dim} (\pi^{-1} (\mathrm{span}(\mathcal F ')))- \mathrm{dim}(I)}{ 2 \mathrm{dim} (\pi^{-1} (\mathrm{span}(\mathcal F '))) - \mathrm{dim}(I)} \Big)
\Big( \frac{\textrm{dim}(W)}{\textrm{dim} (W)- \textrm{dim} (I)}\Big)  \le 1, $$ giving the result.
\end{proof}

Next we show that the only situation where algebraic amenability and proper algebraic amen\-ability differ is when the $\mathbb K$-algebra contains a non-zero
left ideal of finite $\mathbb K$-dimension, as demonstrated by the following theorem. This situation is similar to what is known for 
Hilbert space operators (cf., \cite[Theorem~4.1]{LledoYakubovich13}).
  
\begin{theorem}\label{theorem:when-automatically-proper-alg-amen}
 Let $\cA$ be an infinite dimensional $\mathbb{K}$-algebra that is algebraically amenable but not properly algebraically amenable. Then there exists
a nonzero element $a \in \cA$ with 
\[ 
\dim(\cA \cdot a) < \infty\;.
\]
\end{theorem}

\begin{proof}
Since the algebra $\cA$ is fixed we will denote for simplicity 
the collection $\mathrm{\mbox{F\o l}}(\cA,\mathcal{F},\varepsilon)$ of F\o lner $(\cF,\varepsilon)$-subspaces of $\cA$
by $\mathrm{\mbox{F\o l}}(\mathcal{F},\varepsilon)$.
Since $\cA$ is algebraically amenable, we know that for any 
$\varepsilon > 0$ and any finite set $\mathcal{F} \subset \cA$ the collection
$\mathrm{\mbox{F\o l}}(\mathcal{F}, \varepsilon) \not= \varnothing$. 
Hence we may define 
 \[
  N_{\mathcal{F}, \varepsilon} := \sup \{ \dim(W) \ | \ W \in \mathrm{\mbox{F\o l}}(\mathcal{F}, \varepsilon) \} \in \N \cup \{\infty\} \;.
 \]
On the other hand, as $\cA$ is not properly algebraically amenable, by condition (2) of Proposition~\ref{pro:proper-alg-amen2}, there exist
$\varepsilon_0 > 0$ and finite set $\mathcal{F}_0 \subset \cA$ such that $N_{\mathcal{F}_0, \varepsilon_0} < \infty$. Since $N_{\mathcal{F},
\varepsilon}$ is increasing with respect to $\varepsilon$, without loss of generality we may assume that $\varepsilon_0 \cdot N_{\mathcal{F}_0,
\varepsilon_0} < 1$. 
 
 For any $\varepsilon \in (0, \varepsilon_0]$ and finite set $\mathcal{F} \subset \cA$ containing $\mathcal{F}_0$, we claim that 
 \[
  \mathrm{\mbox{F\o l}}( \mathcal{F}, \varepsilon) = \mathrm{\mbox{F\o l}}(\mathcal{F}, 0) \;.
 \]
 Indeed, the inclusion $\supseteq$ is clear. On the other hand, for any 
 $W \in \mathrm{\mbox{F\o l}}( \mathcal{F}, \varepsilon)$ and $a \in \mathcal{F}$, we have 
 \begin{eqnarray*}
  \dim(a W + W) &\leq & (1 + \varepsilon)  \dim(W) \;\leq\; \dim(W) + \varepsilon  N_{\mathcal{F}, \varepsilon} \\
                &\leq & \dim(W) + \varepsilon_0 \cdot N_{\mathcal{F}_0, \varepsilon_0} \;<\; \dim(W) + 1 \;.
 \end{eqnarray*}
 Since $\dim(a W + W) \geq \dim(W)$ and from the fact that dimensions are in $\N_0$ we conclude that $\dim(aW + W) = \dim(W)$. 
 
 Observe that a non-zero finite-dimensional linear subspace $W$ of $\cA$ is $(\mathcal{F}, 0)$-F\o lner iff $\mathcal{F} \cdot W \subset W$.  
 For any finite set $\mathcal{F} \subset \cA$ 
 containing $\mathcal{F}_0$, since by what we have shown, $\{ \dim(W) \ | \ W \in \mathrm{\mbox{F\o l}}( \mathcal{F}, 0) \}$ is a non-empty
finite subset of $\N$, we have 
 \[
  \mbox{F\o l}_{\mathrm{max}}(\mathcal{F},0) := \{ W \in \mathrm{\mbox{F\o l}}( \mathcal{F}, 0) \ |\ \dim(W) \geq \dim(W'), \ 
  \forall W' \in \mathrm{\mbox{F\o l}}( \mathcal{F}, 0) \}
 \]
 is not empty. Furthermore for any finite set $\mathcal{F}' \subset \cA$ containing $\mathcal{F}$, and for any 
 $W \in \mbox{F\o l}_{\mathrm{max}}(\mathcal{F},0)$ and $W' \in \mbox{F\o l}_{\mathrm{max}}(\mathcal{F}',0)$, 
 we claim that $W' \subseteq W$. Indeed, if this were not the case, then $W + W'$ would be a member of $\mathrm{\mbox{F\o l}}( \mathcal{F}, 0)$ 
with dimension strictly greater than $\dim(W)$, contradicting the definition of $\mbox{F\o l}_{\mathrm{max}}(\mathcal{F},0)$. 
Notice that by setting $\mathcal{F}' = \mathcal{F}$, this claim implies that $ \mbox{F\o l}_{\mathrm{max}}(\mathcal{F},0) $ contains only one element,
which we now denote as $W_\mathcal{F}$. 
 
Consider the decreasing net $\{ \dim (W_\mathcal{F}) \}_{\mathcal{F} \in \mathcal{J}} $ indexed by 
\[ 
 \mathcal{J} := \{\mathcal{F} \subset \cA \ : |\mathcal{F}| < \infty, \ \mathcal{F}_0 \subset \mathcal{F} \}\;.
\]
 Since its range is contained in the finite set $\Z \cap [1, \dim(W_{\mathcal{F}_0})]$, we see that 
 $\displaystyle \lim_{\mathcal{F} \in \mathcal{J}} \dim(W_\mathcal{F})$ exists and is realized by 
 some member $W_{\mathcal{F}_1}$. It follows that $W_\mathcal{F} = W_{\mathcal{F}_1}$ for any finite $\mathcal{F} \subset \cA$ containing
$\mathcal{F}_1$, and thus 
 $a \cdot W_{\mathcal{F}_1} \subseteq W_{\mathcal{F}_1}$ 
 for any $a \in \cA$, i.e., $W_{\mathcal{F}_1}$ is a non-zero left ideal with finite $\mathbb {K}$-dimension. Consequently, if we pick any $a \in
W_{\mathcal{F}_1}$, then 
\[
\dim(\cA \cdot a) \le \dim(W_{\mathcal{F}_1}) < \infty 
\]
and the proof is concluded.
\end{proof}

\begin{corollary}\label{cor:non-zero}
 Let $\cA$ be a $\mathbb K$-algebra without zero-divisor, then $\cA$ is algebraically amenable if and only if it is properly algebraically amenable.
\end{corollary}

\begin{proof}
We only need to prove the case when $\cA$ is infinite-dimensional. Since $\cA$ has no zero-divisor, for any non-zero $a \in \cA$ and finite subset 
$\cF \subset \cA$, we have $$\dim(\mathrm{span}(\cF ) \, a) = \dim(\mathrm{span}(\cF))\;.$$ 
This clearly contradicts the conclusion of 
Theorem~\ref{theorem:when-automatically-proper-alg-amen}, and thus its hypothesis cannot hold.
\end{proof}

\begin{corollary}
 \label{cor:unitizationalgamen-notpaa}
 Suppose that $\cA$ is a non-algebraically amenable algebra such that its unitization $\widetilde{\cA}$ is algebraically amenable.
 Then $\cA$ is a unital algebra. 
 \end{corollary}

\begin{proof}
 By Proposition~\ref{pro:amenability-unitalization} (2), $\widetilde{\cA}$ is not properly algebraically amenable and so, by 
Theorem~\ref{theorem:when-automatically-proper-alg-amen},
 $\widetilde{\cA}$ contains a nonzero finite-dimensional left ideal $I$. Since $\cA$ is not algebraically amenable, we must have $I\cap \cA = \{0\}$, and 
 it follows that $I$ is one-dimensional and that $I\oplus \cA=  \widetilde{\cA }$. Let $(b,1)\in \widetilde{\cA}$, where $b\in \cA$. Then $(a,0)(b,1) \in I$ implies
 that $a(-b)= a$ for all $a\in \cA$, so that $e:= -b$ is a right unit for $\cA$. In particular, $e$ is idempotent and $\cA= \cA e$. If 
 $$(1-e) \cA  = \{ a-ea : a\in \cA \}$$
 is nonzero, then any nonzero finite-dimensional 
 linear subspace of $(1-e) \cA $ is an $(\mathcal F , 0)$-F\o lner subspace for every finite subset $\mathcal F$ of $\cA$, and so $\cA$ is algebraically amenable, contradicting our assumption. 
 Therefore $(1-e)\cA = 0$  and $\cA$ is unital with unit $e$. 
\end{proof}

\begin{example} {\rm (\cite[Corollary 4.5]{Bart})}\label{ex:group-algebra}
The group algebra $\mathbb{K}G$ is algebraically amenable if and only if it is properly algebraically amenable if and only if 
$G$ is amenable.
\end{example}

\section{Paradoxical decompositions and invariant dimension measures of $\mathbb{K}$-algebras} \label{sec:dicho-alg-amenable}

Elek showed that, analogous to the situation for groups, there is a dichotomy between algebraic amenability and a certain kind of 
paradoxical decomposition defined for algebras (cf., \cite[Theorem 2]{Elek03}). However, in his paper, the conditions of countable 
dimensionality and the non-existence of zero-divisors are required. 

We remark here that these conditions can be removed if one replaces Elek's definition (corresponding to \emph{proper} 
algebraic amenability as in Definition~\ref{def:alg-amenable2}~(ii)) with 
algebraic amenability as in Definition~\ref{def:alg-amenable2}~(i). By 
Theorem~\ref{theorem:when-automatically-proper-alg-amen}
the assumption of no zero-divisors happens to have the effect that the properness for algebraic amenability comes for free. 
We will state and prove this general version of Elek's theorem below.

We recall some definitions, adapted to our needs. When working with a zero-divisor $r$, it is useful to restrict 
attention to subspaces $A$ where $r$ acts non-degenerately. More precisely, if $A$ is a linear subspace of $\mathcal A$, 
we say that $r|_A$ is injective if the map $a\mapsto ra$ given by left multiplication by $r$ is injective on $A$. 
Equivalently, $A\cap \mathrm{r.ann}(r) = \{ 0 \}$, where
$$\mathrm{r.ann} (r) = \{ x\in \mathcal A : rx=0 \}$$
is the right annihilator of $r$.

The following definition of paradoxicality is equivalent to the one given by Elek in \cite{Elek03}. 
We prefer this formulation because it is formally closer to the usual condition for actions of groups, 
(cf., \cite[Definition 1.1]{Wagon94}).

\begin{definition}
 \label{def:paradoxicality}
Let $\cA$ be a $\mathbb K$-algebra. Let $\{ e_i \}_{i \in I}$ be a basis of $\cA$ over $\mathbb K$ and $\mathcal{S}$ a subset of $\cA$. A {\it paradoxical decomposition} of $\{ e_i \}_{i \in I}$ by $\mathcal{S}$ consists of two partitions $(L_0, L_1, \ldots, L_n)$ and $(R_0, R_1, \ldots, R_m)$ of $\{ e_i \}_{i \in I}$, i.e.
$$\{ e_i \}_{i \in I} = L_0\sqcup L_1\sqcup \ldots \sqcup L_n = R_0\sqcup R_1 \sqcup \ldots \sqcup R_m \; ,$$
together with elements $g_1,\ldots ,g_n, h_1,\ldots , h_m \in \mathcal{S}$, such that
$$L_0 \cup g_1L_1 \cup \ldots \cup g_nL_n \cup R_0 \cup h_1R_1 \cup \ldots \cup h_mR_m$$
is a disjoint union and linearly independent family in $\cA$.

If such a paradoxical decomposition exists, we say $\{ e_i \}_{i \in I}$ is \emph{paradoxically decomposed} by $\mathcal{S}$.
\end{definition}

Note that, in particular, $g_i|_{A_i}$ and $h_j|_{B_j}$ are injective, where
$A_i$ is the linear span of $L_i$ and $B_j$ is the linear span of $R_j$.

\begin{remark}\label{rem:equiv-Elek-def}
\begin{enumerate}
 \item[(i)] The slight formal inhomogeneity with $L_0$ and $R_0$ can be fixed by adding the unit $\1_\cA$ into $\mathcal{S}$, when $\cA$ is unital. 
 This way, we may write $L_0$ as $\1_\cA L_0$, and $R_0$ as $\1_\cA R_0$. When $\cA$ is not unital, we can still fix it by considering $\mathcal{S}$ 
 as a subset of $\widetilde{\cA}$ and adding $\1_{\widetilde{\cA}}$ into it.
 \item[(ii)] Following \cite[Definition~1.2]{Elek03}, we may also present a variant of the above definition involving only one partition. Namely, we define a \emph{one-partition paradoxical decomposition} of $\{ e_i \}_{i \in I}$ by $\mathcal{S}$ so that it consists of a partition
$\{ e_i \} _{i \in I} = T_1 \sqcup \ldots \sqcup T_k $
and elements $g_1,\ldots ,g_k, h_1,\ldots , h_k \in \mathcal{S}$ with the property that 
\[
 g_1T_1 \cup  \ldots \cup g_k T_k \cup h_1T_1 \cup  \ldots \cup  h_kT_k
\]
is a disjoint union and linearly independent family in $\cA$. Though this is seemingly a more restrictive notion, the existence of this one-partition version is equivalent to that of a general paradoxical decomposition, provided that $\mathcal{S}$ contains the unit (of $\cA$ or $\widetilde{\cA}$). Indeed, starting from a general paradoxical decomposition 
\[
\big((L_0, \ldots, L_n), (R_0, \ldots, R_m), (g_1,\ldots ,g_n), (h_1,\ldots , h_m) \big) \; ,
\]
we may define a one-partition paradoxical decomposition by setting $T_{ij}: = L_i\cap R_j$, $g_{ij}:= g_i$, and $h_{ij}:= h_j$ for $i=0,\ldots , n$ and $j=0,\ldots , m$, with the understanding that $g_0 = h_0 = \1_\cA$ or $\1_{\widetilde{\cA}}$. 
 \item[(iii)] The relation to Elek's definition in \cite{Elek03} is thus as follows: a unital countably dimensional algebra is \emph{paradoxical} in the sense of \cite[Definition~1.2]{Elek03} if and only if for any (countable) basis $\{ e_i \}_{i \in I}$ of $\cA$, there is a paradoxical decomposition of $\{ e_i \}_{i \in I}$ by $\cA$.  
\end{enumerate}
 \end{remark}

The following lemma generalizes \cite[Lemma~2.2]{Elek03}.

\begin{lemma}
 \label{lem:local-doubling}
 Fix $\lambda > 1$. Then a $\mathbb K$-algebra $\cA$ is not algebraically amenable if and only if there exists a finite subset $\cF \subset \cA$, such that for any nonzero finite dimensional linear subspace $W \subset \cA$, we have 
 \[
   \frac{\dim(\cF W +W)}{\dim(W)} > \lambda \;.
 \]
\end{lemma}

\begin{proof}
 By inverting the condition in Remark~\ref{rmk:charac-alg-amen2}, we see that $\cA$ is not algebraically amenable if and only if there exists $\varepsilon > 0$ and finite subset $\cF \subset \cA$, such that for any nonzero finite dimensional linear subspace $W \subset \cA$, we have 
 \[
   \frac{\dim(\cF W +W)}{\dim(W)} > 1+\varepsilon\;.
 \]
 This proves the ``if'' part. For the ``only if'' part, we observe that $\varepsilon$ can be taken to be arbitrarily large: we set
 \[
  \cF^{(n)} = \big\{a_1 \cdots a_m \ | \ m \in \{1, \ldots, n\}, \ \ a_k \in \cF_0, \ \forall k \in \{1, \ldots, m\} \big\} \; .
 \]
 Then by induction we have 
 \[
   \frac{\dim(\cF_0^{(n)} W +W)}{\dim(W)} > (1+\varepsilon)^n \;.
 \]
 For our purpose, we fix $\cF' = \cF^{(\lceil \log_{1+\varepsilon} \lambda \rceil +1)}$, so that 
 \[
   \frac{\dim(\cF' W +W)}{\dim(W)} > \lambda \;.
 \]
 Replacing $\cF$ by $\cF'$ proves the ``only if'' direction. 
\end{proof}

The following is a key proposition of this section. It generalizes Proposition~2.2 in 
\cite{Elek03} to arbitrary $\mathbb K$-algebras which may have zero-divisors, have no unit, or have uncountable dimensions.
To prove this, we adapt ideas from \cite[Theorem 3.4, (vi) $\Rightarrow$ (v)]{KL15} (see also \cite{KL15}) in the context of groups and metric spaces 
to the algebraic setting. 

\begin{proposition}
\label{prop:non-amenable-implies-paradoxic}
Assume that $\cA$ is a $\mathbb K$-algebra which is not algebraically amenable. Then there exists a finite subset $\cF \subset \cA$ such that for any basis $\{e_i\}_{i \in I}$ of $\cA$, there is a paradoxical decomposition of $\{e_i\}_{i \in I}$ by $\cF$.
\end{proposition}

\begin{proof}
 By Lemma~\ref{lem:local-doubling}, there exists a finite subset $\cF \subset \cA$, such that for any nonzero finite dimensional linear subspace $W \subset \cA$, we have 
 \[
   \frac{\dim(\cF W +W)}{\dim(W)} > 2\;.
 \]
 Such a local doubling behavior of $\cF$ can be seen as a local form of paradoxicality, which we will now exploit to produce a paradoxical decomposition for 
 any basis $\{e_i\}_{i \in I}$ of $\cA$. To this end, we define $\cF^+ = \cF \sqcup \{*\}$, where $*$ is an abstract element, for which we prescribe a multiplication 
 $* \cdot e_i = e_i$ for any $i \in I$ (thus $*$ behaves like a unit). Define $\Omega$ to be the set of maps $\omega \colon I \times \{0,1\} \to \mathcal{P}(\cF^+)$ 
 (the power set of $\cF^+$) with the property that for any finite subset $K \subset I \times \{0,1\}$, 
 \[
  \dim_{\mathbb K} \Bigg( \mathrm{span}_{\mathbb K} \bigg( \bigcup_{(i, j) \in K} \bigcup_{a \in \omega(i,j)} a \cdot e_i \bigg) \Bigg) \geq | K | \; .
 \]
 Notice that $\Omega$ is nonempty: the constant function with value $\cF^+$ lives in $\Omega$ because of the local doubling behavior of $\cF$.

 Our goal is to ``trim down'' the above constant set-valued function to a singleton-valued function in $\Omega$. For this purpose, 
 we use the natural partial order on $\Omega$ given by pointwise inclusion: $\omega \leq \omega'$ if $\omega(i,j) \subset \omega'(i,j)$ 
 for any $(i,j) \in I \times \{0,1\}$. Since any descending chain in $\Omega$ has a non-empty lower bound given by pointwise intersection, 
 by Zorn's Lemma, we can find a minimal element $\omega_0 \in \Omega$.
 
 We claim that $| \omega_0(i,j) | = 1$ for any $(i,j) \in I \times \{0,1\}$. Firstly, since 
 \[
  \dim_{\mathbb K} \bigg( \mathrm{span}_{\mathbb K} \Big( \bigcup_{a \in \omega_0(i,j)} a \cdot e_i \Big) \bigg) \geq |\{(i,j)\} | = 1 \; ,
 \]
 we only need to show $| \omega_0(i,j) | \leq 1$. Then, suppose this were not the case: then there exists an index $(i,j)\in I\times\{0,1\}$ and two
 distinct elements $a_0, a_1 \in \omega_0(i,j)$. Notice that the minimality of $\omega_0$ implies that for $l \in \{0,1\}$, 
 we can find a finite subset $K_l \subset I \times \{0,1\}$ not containing $(i,j)$, such that 
 \[
  \dim_{\mathbb K} \Bigg( \mathrm{span}_{\mathbb K} \bigg( \Big( \bigcup_{(i', j') \in K_l} \bigcup_{a \in \omega_0(i',j')} a \cdot e_{i'} \Big) \cup \Big( \bigcup_{a \in \omega_0(i,j) \setminus \{a_l\} } a \cdot e_i \Big) \bigg) \Bigg) 
  \leq |K_l| \; ,
 \]
 since otherwise if no such $K_l$ exists, we would be able to remove $a_l$ from $\omega_0(i,j)$ to produce a new element in $\Omega$ strictly smaller than $\omega_0$. 
 
 Now because of the simple fact that $(\omega_0(i,j) \setminus \{a_0\} ) \cup (\omega_0(i,j) \setminus \{a_1\} ) = \omega_0(i,j)$, we would see that, if we denote 
 \[
  W_l := \mathrm{span}_{\mathbb K} \bigg( \Big( \bigcup_{(i', j') \in K_l} \bigcup_{a \in \omega_0(i',j')} a \cdot e_{i'} \Big) \cup \Big( \bigcup_{a \in \omega_0(i,j) \setminus \{a_l\} } a \cdot e_i \Big) \bigg)
 \]
 for $l \in \{0,1\}$, then 
 \begin{eqnarray*}
   | K_0| + | K_1 | 
         &\geq&   \dim_{\mathbb K} (W_0) +  \dim_{\mathbb K} (W_1) \\
         &=&      \dim_{\mathbb K} (W_0 + W_1) +  \dim_{\mathbb K} (W_0 \cap W_1) \\
         &\geq&  \dim_{\mathbb K} \Bigg( \mathrm{span}_{\mathbb K} \bigg( \Big( \bigcup_{(i', j') \in K_0 \cup K_1} \bigcup_{a \in \omega_0(i',j')} a \cdot e_{i'} \Big) 
                 \cup \Big( \bigcup_{a \in \omega_0(i,j) } a \cdot e_i \Big) \bigg) \Bigg) \\
         && + \dim_{\mathbb K} \Bigg( \mathrm{span}_{\mathbb K} \Big( \bigcup_{(i', j') \in K_0 \cap K_1} \bigcup_{a \in \omega_0(i',j')} a \cdot e_{i'} \Big) \Bigg) \\
         &\geq &|K_0 \cup K_1 \cup \{(i,j)\}| + |K_0 \cap K_1| \\
         &=& |K_0 \cup K_1| + 1 + |K_0 \cap K_1| \\
         &=& |K_0| + |K_1| + 1 \; , 
 \end{eqnarray*}
 which gives a contradiction. Hence we have proved our claim that $| \omega_0(i,j)| = 1$ for any $(i,j) \in I \times \{0,1\}$.
 
 Thus we may define $\phi \colon I \times \{0,1\} \to \cF^+$ such that $\omega_0(i,j) = \{ \phi(i,j) \}$. It follows from the defining property of $\Omega$ that $\phi$ satisfies 
 \[
  \dim_{\mathbb K} \bigg( \mathrm{span}_{\mathbb K} \Big( \bigcup_{(i, j) \in K} \phi(i,j) \cdot e_i \Big) \bigg) = | K | 
 \]
 for any finite subset $K \subset I \times \{0,1\}$, i.e.,  $\{ \phi(i,j) \cdot e_i \}_{(i,j) \in I \times \{0,1\} }$ is a linearly independent family in $\cA$. 
 
 To conclude the proof, we define, for each $a \in \cF^+$, 
 \begin{align*}
  L_a =& \ \{ e_i \ | \ i \in I, \phi(i,0) = a \} \\
  R_a =& \ \{ e_i \ | \ i \in I, \phi(i,1) = a \} \; .
 \end{align*}
 Therefore we have two finite partitions 
 \[
  \{e_i\}_{i\in I} = L_* \sqcup \bigsqcup_{a \in \cF} L_a = R_* \sqcup \bigsqcup_{a \in \cF} R_a 
 \]
 such that 
 \[
  \Big( L_* \cup \bigcup_{a \in \cF} a L_a \Big) \cup  \Big( R_* \cup \bigcup_{a \in \cF} a R_a \Big) 
 \]
 is a disjoint union and linearly independent family in $\cA$. Thus we have produced a paradoxical decomposition of $\{e_i\}_{i\in I}$ by $\cF$ in the sense of Definition~\ref{def:paradoxicality}.
\end{proof}

Now we define a suitable notion of invariant dimension-measure for $\mathbb{K}$-algebras, an analogue of invariant mean for amenable groups. Note that the lack of distributivity in the lattice of subspaces of a vector space
makes it necessary to give up some of the properties one would expect for this concept.

\begin{definition}
 \label{def:dimension-measure}
 Let $\cA$ be a $\mathbb K$-algebra and $\{e_i \}_{i \in I}$ be a $\mathbb K$-linear basis of $\cA$. A {\it dimension-measure} on $\cA$ associated to $\{e_i \}_{i \in I}$ is a function $\mu$ from the 
 set of linear subspaces of $\cA$
 to $[0,1]$ which satisfies the following properties:
 \begin{enumerate}
  \item[(i)] $ \mu (\cA) =1$.
  \item[(ii)] If $A,B$ are linear subspaces in $\cA$ with $A\cap B = \{0 \}$, then $\mu (A\oplus B)  \ge \mu (A) +  \mu (B)$.
  \item[(iii)] For every partition $L_1\sqcup L_2\sqcup \ldots \sqcup L_m$
  of $\{e_i \}_{i \in I}$, we have $\sum _{k=1}^m \mu (\mathrm{span}(L_k)) =1$. 
 \end{enumerate}
 Let $\mathcal{S}$ be a subset of $\cA$. We say $\mu$ is \emph{$\mathcal{S}$-invariant} if 
 \begin{enumerate}
  \item[(iv)] For any $s\in \mathcal{S}$ and any linear subspace $A \subset \cA$ such that $s|_A$ is injective, we have $\mu (sA) \ge \mu (A)$.
 \end{enumerate}
\end{definition}

Note that if $\mu$ is a dimension-measure on $\cA$ and $A\subseteq B$ are subspaces of $\cA$, then, by property (ii), it follows that 
$\mu (A) \le \mu (B)$.

We can now state the following generalization of \cite[Theorem 1]{Elek03}.

\begin{theorem}
\label{theorem:aa-characterized}
Let $\cA$ be a $\mathbb{K}$-algebra. Then the following conditions are equivalent:
 \begin{enumerate}
  \item $\cA$ is algebraically amenable.
  \item For any finite subset $\cF \subset \cA$, there is a basis of $\cA$ that cannot be paradoxically decomposed by $\cF$.
  \item For any countably dimensional linear subspace $W \subset \cA$, there is a basis of $\cA$ that cannot be paradoxically decomposed by $W$.
  \item For any countably dimensional linear subspace $W \subset \cA$, there exists a $W$-invariant dimension-measure on $\cA$ (associated to some basis).
 \end{enumerate}
\end{theorem}

\begin{proof}
The implication (2)$\Rightarrow$(1) follows from Proposition~\ref{prop:non-amenable-implies-paradoxic}. The implication (3)$\Rightarrow$(2) is immediate by setting $W = \mathrm{span}(\cF)$. 

To show (4)$\Rightarrow$(3), we fix an arbitrary countably dimensional linear subspace $W \subset \cA$. By (4), there is a basis $\{ e_i \}_{i \in I}$ of $\cA$ and a $W$-invariant dimension-measure $\mu$ on $\cA$ associated to $\{ e_i \}_{i \in I}$. Suppose there were a paradoxical decomposition 
\[
 \big( (L_0, \ldots, L_n), (R_0, \ldots, R_m), (g_1, \ldots, g_n), (h_1, \ldots, h_m) \big)
\]
of $\{ e_i \}_{i \in I}$ by $W$. Put $A_k := \textrm{span}(L_k)$ and $B_l := \textrm{span}(R_l)$. 
We have $\sum_{k=0}^n \mu (A_k) = 1 = \sum_{l=0}^m \mu (B_l)$ (by (iii) in Definition~\ref{def:dimension-measure}).
Also $g_k|_{A_k}$ and $h_l|_{B_l}$ are injective for all $k,l$ and so $\mu (g_kA_k )\ge \mu (A_k) $ and $\mu (h_lB_l) \ge \mu (B_l)$ 
for all $k,l$ (by (iii)), so that we get 
\begin{align*}
1 & \ge \mu (A_0 \oplus g_1A_1\oplus \ldots g_nA_n\oplus B_0 \oplus h_1B_1\oplus \ldots \oplus h_mB_m)\\
& \ge \mu (A_0) + \sum_{k=1}^n \mu (g_kA_k) + \mu (B_0) + \sum_{l=1}^m \mu (h_lB_l)\\
&  \ge \sum_{k=0}^n \mu (A_k) + \sum_{l=0}^m \mu (B_l) = 2,  
\end{align*}
which is a contradiction.

Finally, to show (1)$\Rightarrow$(4) we construct, for an arbitrary countably dimensional linear subspace $W \subset \cA$, a dimension-measure $\mu$ on $\cA$ associated to some basis. This involves two cases:
\begin{enumerate}
 \item[Case 1:]$\cA$ is properly algebraically amenable. By Proposition~\ref{pro:alg-amenability-countability}, there is a countably dimensional subalgebra $\cB \subset \cA$ that is properly algebraically amenable and contains $W$. Let $\{ W_i \}_{i=1}^{\infty}$ be an increasing 
sequence of finite-dimensional subspaces of $\cA$ such that $\cB= \cup_{i=1}^{\infty} W_i$, and such that 
$$\lim_{i\to\infty} \frac{\dim (aW_i+W_i)}{\dim (W_i)} = 1$$
for all $a\in \cB$. Let $\omega $ be a free ultrafilter on $\mathbb N$, and let $\{e_i \}_{i=1}^{\infty}$ be a basis for $\cB$ 
obtained by successively enlarging basis of the spaces $W_i$ (cf. \cite[Proposition 2.1]{Elek03}). We then enlarge  $\{e_i \}_{i=1}^{\infty}$ to a basis $\{e_i \}_{i \in I}$ of $\cA$, where $\mathbb{N} \subset I$.
For a linear subspace $A$ of $\cA$, set
\[
\mu (A) = \lim _{\omega} \frac{\dim (A\cap W_i)}{\dim (W_i)}\;.
\]
Obviously, we have $\mu (\cA) =1 $ and $0 \le \mu (A)\le 1$ for every subspace $A$. Moreover, properties~(ii) and (iii) in 
Definition~\ref{def:dimension-measure} clearly hold, so we only need to check (iv). 

To prove (iv) we first show that for any $a\in W$ and any linear subspace $A$ we have
\begin{equation}\label{eq:mu-limit}
\mu (A) = \lim _{\omega}\frac{\dim ( (W_i+aW_i)\cap A ) } {\dim (W_i)}\;.
\end{equation}

 Write $T_i = (W_i+aW_i)\cap A$. Then $T_i\cap W_i = A\cap W_i $, so that $T_i = (W_i\cap A)\oplus T_i'$ with
 $T_i'\cap W_i = \{ 0\}$. Hence 
 $$\frac{\dim (T_i)}{\dim (W_i)} = \frac{\dim (W_i\cap A)}{\dim (W_i)} + \frac{\dim (T_i')}{\dim (W_i)}.$$
 Since $\dim (T_i') /\dim (W_i)\to 0$, we obtain the result.

We now show (iv). Let $a \in W$ be such that $a|_A$ is injective. Then we have
\begin{align*}
 \mu (aA) & = \lim _{\omega} \frac{\dim ((W_i+aW_i)\cap aA)}{\dim (W_i)} \ge \lim _{\omega} \frac{\dim (aW_i\cap aA)}{\dim (W_i)}\\
& \ge \lim _{\omega} \frac{\dim (a(W_i\cap A))}{\dim (W_i)}
 = \lim _{\omega} \frac{\dim (W_i \cap A)}{\dim (W_i)}= \mu (A) ,
\end{align*}
where in the second equality we have used that $a|A$ is injective.

\item[Case 2:]$\cA$ is algebraically amenable but not properly algebraically amenable.  
By Theorem~\ref{theorem:when-automatically-proper-alg-amen}, we only need to build a dimension-measure 
in the case where $\cA$ has a nonzero finite-dimensional left ideal $I$. 
This is easily taken care of by defining $$\mu (A) = \frac{\dim (I\cap A)}{\dim (I)}$$ for
each linear subspace $A \subset \cA$. 
\end{enumerate}
This concludes the proof of Theorem.
\end{proof}

For countably dimensional (or equivalently, countably generated) $\mathbb{K}$-algebras, the statement of the previous theorem can be somewhat simplified:

\begin{corollary}\label{cor:charac-amen-countable}
 Let $\cA$ be a countably dimensional $\mathbb{K}$-algebra. Then the following conditions are equivalent:
 \begin{enumerate}
  \item $\cA$ is algebraically amenable.
  \item There is a basis of $\cA$ that cannot be paradoxically decomposed by $\cA$.
  \item There exists an $\cA$-invariant dimension-measure on $\cA$ (associated to some basis).
 \end{enumerate}
\end{corollary}

\begin{proof}
 This is immediate after we set $W = \cA$ in the statement of Theorem~\ref{theorem:aa-characterized}.
\end{proof}

\begin{remark}\label{rem:NZD}
 If $\mu$ is as build before, and $a$ is a non-zero-divisor in $\cA$, then one gets $\mu (aA)= \mu (A)$ (cf., \cite{Elek03}). 
 The reason is that, in this case, we have
 $$\dim (a^{-1}W_i + W_i ) \le  \dim (W_i + aW_i), $$
 where $a^{-1}W_i= \{ x\in \cA : ax\in W_i \}$, because left multiplication by $a$ induces an injective map from $a^{-1}W_i+W_i$ into
 $W_i+aW_i$.  Therefore we get
 $$\lim_i \frac{\dim (a^{-1}W_i+W_i)}{\dim (W_i)} = 1.$$
 Hence, for any linear subspace $A$ of $\cA$, we can show
 $$\mu (A) = \lim_{\omega}  \frac{\dim ((a^{-1}W_i+W_i)\cap A)}{\dim (W_i)}$$
 just as in the proof of Eq.~(\ref{eq:mu-limit}).
 
 Moreover, we have
 $$\frac{\dim (a(a^{-1}W_i))}{\dim (W_i)} \ge \frac{\dim (W_i \cap aW_i)}{\dim (W_i)} \to 1 $$
 and thus, 
 $$\mu (B) = \lim _{\omega} \frac{\dim (a(a^{-1}W_i) \cap B)}{\dim (W_i)}$$
 for any linear subspace $B$ of $\cA$.
 We obtain
\begin{align*} \mu (A) & = \lim _{\omega} \frac{\dim ((a^{-1}W_i+W_i)\cap A)}{\dim (W_i)}
 \ge \lim _{\omega} \frac{\dim (a^{-1}W_i\cap A)}{\dim (W_i)}\\
& = \lim _{\omega} \frac{\dim (a(a^{-1}W_i)\cap aA)}{\dim (W_i)} =  \mu (aA).
\end{align*}
This proves our claim. \qed
\end{remark}

Recall the usual Murray-von~Neumann equivalence $\sim $ and comparison $\gtrsim$ for idempotents of an algebra, defined as follows: for idempotents $e,f$ in $\cA$, write $e\sim f$ if there are $x,y\in \cA$ such that $e=xy$ and $f=yx$; write $e \gtrsim f$ if there are $x,y\in \cA$ such that $xy \in e\cA e$ and $f=yx$. These relations naturally extends to the infinite matrix algebra $M_\infty(\cA) := \bigcup_{n=1}^\infty M_n(\cA)$ where the $M_n(\cA)$ embeds into $M_{n+1}(\cA)$ block-diagonally as $M_n(\cA) \oplus 0$. 

An idempotent $e$ in an algebra $\cA$ is said to be {\it properly infinite} if there are orthogonal idempotents $e_1, e_2$ in $e\cA e$ such that $e_1 \sim e\sim e_2$. Equivalently, $e$ is properly infinite if $e \gtrsim e \oplus e$.
A (nonzero) unital algebra $\cA$ is said to be {\it properly infinite} in case $\1$ is a properly infinite idempotent. 

As an application of the dichotomy shown in Theorem~\ref{theorem:aa-characterized}, we present a method of producing non-algebraically amenable $\mathbb{K}$-algebras:

\begin{corollary}
\label{cor:Properly-infinite}
 A properly infinite unital $\mathbb{K}$-algebra is not algebraically amenable.
\end{corollary}
 
 \begin{proof}
 If $\cA$ is properly infinite, it contains elements $u,v,u',v'$ satisfying the relations
 \[
  u u' = v v' = \1_{\cA} \;, \ v u' = 0=uv' \; .
 \]
Suppose that there exists a $\{u,u',v,v'\}$-invariant dimension measure on $\cA$ (associated to some basis). Notice that the first 
set of identities imply that $u'|_\cA$ and $v'|_\cA$ are injective. Thus by invariance, we have
 \[
 1= \mu(\cA) \leq \mu(u'\cA) \leq 1 \; ,
 \]
 which implies $\mu(u'\cA) = \mu(\cA) = 1$, and similarly $\mu(v'\cA) = \mu(\cA) = 1$. On the other hand, for any $a, b \in \cA$ with $u'a = v' b$, we have $b = v v' b = v u' a = 0$ by the second identity. It follows that $u' \cA \cap v' \cA = 0$, and thus $\mu(u' \cA + v' \cA) \geq \mu(u'\cA) + \mu(v'\cA) = 2$, 
 which is an impossible value for $\mu$. This proves our claim. 
\end{proof}

\section{Leavitt algebras and Leavitt path algberas}\label{sec:Leavitt}

In this section we study the amenability of Leavitt algebras and Leavitt path algebras (see below for the specific definitions).
Classical Leavitt algebras were invented by Leavitt (\cite{Leav}, \cite{LeavDuke}) to provide universal examples of algebras without the invariant basis number property.
As such, they cannot be algebraically amenable, by a result of Elek \cite[Corollary 3.1(1)]{Elek03}. Leavitt path algebras provide a wide generalization of classical Leavitt algebras, in much the same way as
graph $C^*$-algebras generalize Cuntz algebras (see e.g. \cite{Raeburn} for an introduction to the theory of graph $C^*$-algebras).

\subsection{Leavitt algebras}

Extending results by Aljadeff and Rosset \cite{AR} and Rowen \cite{Rowen}, Elek proved in \cite{Elek03} that any finitely generated 
unital algebraically amenable $\mathbb{K}$-algebra $\cA$ has the {\em Invariant Basis Number} (IBN) property, 
that is, any finitely generated free $\cA$-module has a well-defined rank. This is equivalent to the condition
$$\cA^n\cong \cA^m \text{~as~left~$\cA$-modules} \implies n=m ,$$
for any positive integers $n,m$. 
We will use the observation in Corollary~\ref{cor:Properly-infinite} to obtain a proof 
of the IBN property of general unital amenable algebras.

\begin{definition}
\label{def:LeavittAlgebras}
Let $\mathbb{K}$ be a field. 
\begin{enumerate}
 \item[(i)] Let $n,m$ be integers such that $1\le m< n$. Then the Leavitt algebra $L(m,n)= L_\mathbb{K}(m,n)$ is the algebra generated by elements $X_{ij}$ and $Y_{ji}$,
for $i=1,\ldots,m$ and $j=1,\ldots,n$, such that $XY = \1_m$ and $YX=
\1_n$, where $X$ denotes the $m\times n$ matrix $(X_{ij})$ and $Y$
denotes the $n\times m$ matrix $(Y_{ji})$.
 \item[(ii)] The algebra $L_{\infty }= L_{\mathbb{K}, \infty}$ is the unital algebra generated by $x_1,y_1,x_2, y_2, \ldots $ subject to the relations 
 $y_jx_i = \delta_{i,j} \1$.
 \end{enumerate}
\end{definition}
 
 The algebras $L(m,n)$ are simple if and only if $m=1$ \cite[Theorems 2 and 3]{LeavDuke}. The algebra $L_{\infty }$ is simple \cite[Theorem 4.3]{AGP}. 
 
 The following is well-known (cf. \cite{Abrams15} or \cite{Leav}):
 
 \begin{proposition} Let $\cA$ be a (nonzero) unital algebra over a field $\mathbb{K}$. 
  \begin{enumerate}
   \item $\cA$  does not satisfy the IBN property if and only if there is a unital homomorphism $L(m,n) \to \cA$ for some $1\le m < n$.
   \item $\cA $ is properly infinite if and only if there is a unital embedding $L_{\infty }\to \cA$.   
  \end{enumerate}
\end{proposition}

\begin{proof}
(1) By definition, if an algebra $\cA$ does not have the IBN property, then
 there are $m,n$ with $1\le m <n$ such that $\cA ^m \cong \cA ^n$, and this isomorphism of free modules will be implemented by matrices $X'\in M_{m\times n}(\cA)$ and $Y' 
 \in M_{n\times m} (\cA )$ such that $X'Y'=I_m$ and $Y'X'= I_n$. We thus obtain a unital homomorphism $L(m,n) \to \cA$. The converse is trivial.

(2) If $\cA$ is properly infinite, we may inductively find an infinite sequence $e_1, e_2,\ldots $ of mutually orthogonal idempotents such that $e_i\sim 1$ for all $i$.
This enables us to define a homomorphism $L_{\infty } \to \cA$ which is injective because $L_{\infty }$ is simple. The converse is obvious. 
 \end{proof}

Note that $L_{\infty}$ is properly infinite but does have the IBN property. 

\begin{proposition}
 \label{prop:amen-implies-IBN} If $\cA$ is a unital algebraically amenable algebra, then $\cA$ has the IBN property.
 \end{proposition}

\begin{proof}
 Suppose that $\cA$ does not have the IBN property. Then there are integers $m,n$ with $1\le m <n$ and there is a unital homomorphism
 $L(m,n) \to  \cA$. Now $M_n(\cA)\cong M_m(\cA)$ is properly infinite, so that by Corollary~\ref{cor:Properly-infinite}, $M_n(\cA )$ is not algebraically amenable.
 If $\cA$ were amenable then $M_n(\cA) \cong \cA \otimes M_n(\mathbb{K})$ would be amenable too (\cite[Proposition 4.3(2)]{Cec-Sam-08}). Therefore $\cA$ is 
 not algebraically amenable, showing the result.
 \end{proof}

 \begin{corollary}\label{cor:Leavitt-alg-amen}
 A unital $\mathbb{K}$-algebra $\cA$ that unitally contains the Leavitt algebra $L(m,n) $ for some $1\le m < n$ is not algebraically amenable. \qed
\end{corollary}

\subsection{Leavitt path algebras}

In general, a non-algebraically amenable algebra need not be properly infinite, as the non-commutative free algebra shows.
We now show that, within a certain class of algebras, the class of {\it Leavitt path algebras}, both properties are indeed equivalent.  
Note that this class of algebras includes the algebras $L(1,n)$ and $L_{\infty }$ as distinguished members. (The algebras $L(m,n)$, with $1<m<n$ are
not included in the class of Leavitt path algebras, but they are Morita-equivalent to Leavitt path algebras associated to separated graphs \cite{AG12}.)
We refer the reader to \cite{Abrams15} and the references therein for more information about Leavitt path algebras.
 
We recall some definitions needed here.

\begin{definition}
 A (directed) graph $E=(E^{0},E^{1},r,s)$ consists of two sets $E^{0}$ and $E^{1}$ together with range and source maps $r,s:E^{1}\rightarrow E^{0}$. The elements of $E^{0}$ are called \emph{vertices} and the elements of $E^{1}$
\emph{edges}. 

 A vertex $v$ is called a \emph{sink} if it emits no edges, that is, $s^{-1}(v)=\emptyset$, the empty set. The vertex $v$ is called a \emph{finite emitter} if $s^{-1}(v)$ is finite; otherwise it is an \emph{infinite emitter}. A finite emitter which is not a sink is also called a \emph{regular vertex}. For each $e\in E^{1}$, we call $e^{\ast}$ a \emph{ghost edge}. We let $r(e^{\ast})$ denote $s(e)$, and we let $s(e^{\ast})$ denote $r(e)$. 
\end{definition}

The Leavitt path algebras are built on top of these directed graphs.

\begin{definition}\label{def:Leavitt-path-alg}
 Given an arbitrary graph $E$ and a field $\mathbb{K}$, the \emph{Leavitt path $\mathbb{K}$-algebra }$L_{\mathbb{K}}(E)$ (or simply $L(E)$) is defined to be the $\mathbb{K}$-algebra generated by a set $\{v:v\in E^{0}\}$ of pairwise orthogonal idempotents together with a set of variables $\{e,e^{\ast}:e\in E^{1}\}$ which satisfy the following conditions:
 \begin{enumerate}
  \item \ $s(e)e=e=er(e)$ for all $e\in E^{1}$.
  \item $r(e)e^{\ast}=e^{\ast}=e^{\ast}s(e)$\ for all $e\in E^{1}$.
  \item (The ``CK-1 relations") For all $e,f\in E^{1}$, $e^{\ast}e=r(e)$ and $e^{\ast}f=0$ if $e\neq f$.
  \item (The ``CK-2 relations") For every regular vertex $v\in E^{0}$,
  \[
   v=\sum_{e\in E^{1},s(e)=v}ee^{\ast}.
  \]
 \end{enumerate}
\end{definition}

In a sense, the definition of a Leavitt path algebra treats the graph as a dynamical system: its multiplication is based on the ways one can traverse
the vertices of the graph via the edges. This naturally brings into the picture notions such as paths and cycles. 

\begin{definition}
 A \emph{(finite) path} $\mu$ of length $n>0$ is a finite sequence of edges $\mu=e_{1}e_{2}\cdot\cdot\cdot e_{n}$ with $r(e_{i})=s(e_{i+1})$ for all $i=1,\cdot\cdot\cdot,n-1$. In this case, $\mu^{\ast}=e_{n}^{\ast}\cdot\cdot\cdot e_{2}^{\ast}e_{1}^{\ast}$ is the corresponding ghost path. The set of all vertices on the path $\mu$ is denoted by $\mu^{0}$. Any vertex $v$ is considered a path of length $0$.

A non-trivial path $\mu$ $=e_{1}\dots e_{n}$ in $E$ is \emph{closed} if $r(e_{n}%
)=s(e_{1})$, in which case $\mu$ is said to be \emph{based at} the vertex $s(e_{1})$. By cyclically permuting the edges of a closed path $\mu=e_{1}\dots e_{n}$, we obtain a closed path $e_{k}\dots e_{n}e_{1}\dots e_{k-1}$ based at the vertex $s(e_{k})$ for any $k = 1, \ldots, n$.
A closed path $\mu$ as above is called \emph{simple} provided it does not
pass through its base more than once, i.e., $s(e_{i})\neq s(e_{1})$ for all
$i=2,...,n$. 

The closed path $\mu$ is called a \emph{cycle based at} $v$ 
if $s(e_1)=v$ and it does not
pass through any of its vertices twice, that is, if $s(e_{i})\neq s(e_{j})$
whenever $i\neq j$. A nontrivial cyclic permutation of a cycle based at a vertex $v$ is then a cycle based at a different vertex. Cyclic permutation thus induces an equivalence relation on the set of all cycles based at vertices. An equivalence class of it is called a \emph{cycle}. Note that it is meaningful to talk about the set of vertices of a cycle, which we denote by $c^0$. 
A cycle $c$ is called an \emph{exclusive cycle} if it
is disjoint with every other cycle; equivalently, no
vertex $v$ on $c$ is the base of a different cycle other than the cyclic permutation of $c$ based at $v$.
\end{definition}

The following lemma was shown in the row-finite case in \cite[Lemma 7.3]{AGPS}. We include the identical proof for completeness.

\begin{lemma}\label{vInfiniteIdempotent} Let $E$ be an arbitrary graph and let $\mathbb{K}$ be a field. 
If $v\in E^0$ belongs to a non-exclusive cycle, then $v$ is a properly
infinite idempotent in $L_\K(E)$.
\end{lemma}

\begin{proof}
 We would like to show that $v \gtrsim v \oplus v$. To this end, let $e_1\dots e_m$ and $f_1\dots f_n$ be two different closed simple paths in $E$ based at $v$. Then there is some positive integer $t$ such that $e_i=f_i$ for $i=1,\ldots,t-1$ while $e_t\ne f_t$. Thus, 
 we have $s(e_{t})= s(f_{t})$ but $e_{t} \ne f_{t}$. We observe
 \[
  v = s(e_1) \gtrsim r(e_1) = s(e_2) \gtrsim \ldots \gtrsim r(e_{t-1}) = s(e_{t}) \; ,
 \]
 and similarly $r(e_t) \gtrsim r(e_m) = v$ and $r(f_t) \gtrsim r(f_n) = v$. Since $e_t e_t^*$ and $f_t f_t^*$ are two mutually orthogonal idempotents below $s(e_{t})$, we have 
 \[
  v \gtrsim s(e_{t})  \gtrsim e_t e_t^* \oplus f_t f_t^* \sim e_t^* e_t \oplus f_t^* f_t = r(e_t)\oplus r(f_t) \gtrsim  v\oplus v \; .
 \]

 Therefore $v$ is properly infinite. 
\end{proof}

Below we summarize some additional basic terminologies and properties for graphs and Leavitt path algebras. For this we follow the book in preparation \cite{AAS}.

\begin{remark}\label{rmk:Leavitt-path-algebras}
 Let $E$ be a directed graph.
 \begin{enumerate}
  \item\label{rmk:Leavitt-path-algebras:preorder} If there is a path from a vertex $u$ to a vertex $v$, we write $u\geq v$. This defines a pre-order
  on $E^0$. As we have shown above, $u\geq v$ implies $u \gtrsim v$ in $L_{\mathbb{K}}(E)$. Since all vertices on a cycle are equivalent with regard
  to the pre-order $\geq$, it induces a pre-order on the set of all cycles, so that for any cycles $c_1$ and $c_2$, we have $c_1\geq c_2$ if and only
  if there is path from a vertex of $c_1$ to a vertex of $c_2$. 
 
 \item\label{rmk:Leavitt-path-algebras:equivalent-cycles} Let $C$ be the set of all cycles in $E$. 
 Let $C/{\sim} $ be the partially ordered set obtained by antisymmetrization of the pre-order $\le $ on $C$, so that $c\sim c'$ if and only if
 $c \le c'$ and $c'\le c$. Note that the exclusive cycles are precisely those cycles $c$ such that $[c]= \{ c \}$, and that $C/{\sim} $ is a finite set if $E$ has a finite number of vertices.
 
 \item\label{rmk:Leavitt-path-algebras:unital} The Leavitt path algebra $L_\K(E)$ is unital if and only if $| E^0 | < \infty$, in which case the unit is given by $\sum_{v \in E^0} v$. 
 
 \item\label{rmk:Leavitt-path-algebras:paths-elements} Every finite path $\mu = e_1 \cdots e_n$ induces the elements $\mu = e_1 \cdots e_n$ and $\mu^{\ast}=e_{n}^{\ast}\cdots e_{1}^{\ast}$ in $L_{\mathbb{K}}(E)$. By a simple induction, we see that the Leavitt path algebra $L_\K(E)$ is linearly spanned by terms of the form $\lambda \rho ^*$, where $\lambda $ and $\rho$ are paths such that $r(\lambda )= r(\rho )$. 
 
 \item\label{rmk:Leavitt-path-algebras:finite-acyclic} The graph $E$ is called \emph{acyclic} if it contains no cycle, and \emph{finite} if both $E^0$ and $E^1$ are finite sets. A finite acyclic graph clearly contains finitely many paths. Thus by (\ref{rmk:Leavitt-path-algebras:paths-elements}), we see that $L_{\mathbb{K}}(E)$ is finite-dimensional. In fact, in this case, $L_{\mathbb{K}}(E)$ is a finite direct sum of matrix algebras over $\K$ (cf., \cite[Theorem~3.1]{Abrams15}). 
 
 \item\label{rmk:Leavitt-path-algebras:hereditary-saturated} A subset $H$ of $E^{0}$ is called \emph{hereditary} if, whenever $v\in H$ and
$w\in E^{0}$ satisfy $v\geq w$, then $w\in H$. A hereditary set is
\emph{saturated} if, for any regular vertex $v$, $r(s^{-1}(v))\subseteq H$
implies $v\in H$. For $X\subseteq E^0$, we denote by $\overline{X}$ the
hereditary saturated closure of $X$. To compute $\overline{X}$, one can first compute the {\it tree}
of $X$, $T(X) := \{ w\in E^0 : w\le v \text{ for some } v\in X \}$, which is the smallest hereditary subset of $E^0$ containing $X$, and then, setting $\Lambda _0(T(X)) := T(X)$, compute inductively
\[
 \Lambda _n(T(X)) := \{y\in E^0_{\mathrm{reg}}  : r(s^{-1}(y))\subseteq \Lambda
_{n-1}(T(X))\} \cup \Lambda _{n-1}(T(X))
\]
for $n = 1, 2 , \ldots$, where $E^0_{\mathrm{reg}}$ is the set of regular vertices. It is easy to see $\ol{X} = \bigcup _{n=0}^{\infty} \Lambda
_n(T(X))$.

 \item\label{rmk:Leavitt-path-algebras:ideals} We shall use the following constructions from \cite{tomforde}. A \emph{breaking vertex }of a hereditary saturated subset $H$
is an infinite emitter $w\in E^{0}\setminus H$ with the property that
$1\leq|s^{-1}(w)\cap r^{-1}(E^{0}\setminus H)|<\infty$. The set of all
breaking vertices of $H$ is denoted by $B_{H}$. For any $v\in B_{H}$, we define an idempotent $v^{H} \in L_{\mathbb{K}}(E)$ by
\[
 v^{H} := v-\sum_{s(e)=v,r(e)\notin H}ee^{\ast} \; .
\]
Given a hereditary saturated subset $H$ and a subset $S\subseteq B_{H}$, $(H,S)$ is
called an \emph{admissible pair.} Given an
admissible pair $(H,S)$, $I(H,S)$ denotes the ideal generated by
$H\cup\{v^{H}:v\in S\}$. Then we have an isomorphism $L_{\mathbb{K}}(E)/I(H,S)\cong L_{\mathbb{K}}%
(E / (H,S))$. Here $E / (H,S)$ is the \emph{quotient graph} of $E$ in which \emph{ }$(E / (H,S))^{0}=(E^{0}\backslash H)\cup
\{v^{\prime}:v\in B_{H}\backslash S\}$ and $(E / (H,S))^{1}=\{e\in
E^{1}:r(e)\notin H\}\cup\{e^{\prime}:e\in E^{1},r(e)\in B_{H}\backslash S\}$
and $r,s$ are extended to $(E / (H,S))^{1}$ by setting $s(e^{\prime
})=s(e)$ and $r(e^{\prime})=r(e)^{\prime}$. 
Thus when $S=B_{H}$, 
we can identify the quotient graph $E \setminus (H,B_{H})$ with the subgraph $E / H$ of $E$, where $(E / H)^0= E^0\setminus H$ and
$(E / H)^1= \{ e\in E^1 : r(e) \notin H \}$. It was shown in \cite{tomforde} that the graded ideals of
$L_{\mathbb{K}}(E)$ are precisely the ideals of the form $I(H,S)$ for some
admissible pair $(H,S)$, though we will not make use of this. 

 \item\label{rmk:Leavitt-path-algebras:full-subgraph} A subgraph $E'$ of $E$ is called \emph{full} if $(E')^1 = \{e \in E^{1}:s(e), r(e) \in (E')^0 \}$. For a subset $X \subset E^0$, we define a full subgraph $M(X)$ so that 
 \[
  M(X)^0=\{w\in E^{0}:w\geq v \text{~for~some~} v \in X \} \; .
 \]
 If $X = \{v\}$ for some $v \in E^0$, we also write $M(v) = M(\{v\})$.  
 Also define 
 \[
  H(v)=E^{0}\setminus M(v)^0 \; ,
 \]
 which is hereditary by design. Note that any edge $e$ is in a cycle if and only if $r(e) \notin H(s(e))$ if and only if $r(e) \in M(s(e))^0$. It
 follows that if $v$ belongs to a cycle, then $H(v)$ is a hereditary saturated subset of $E$.
 \qed
 \end{enumerate}

\end{remark}

\begin{theorem}
 \label{thm:amenLPAs} Let $E$ be a nontrivial directed graph and let $\mathbb{K}$ be a field. Let $H$ be the smallest hereditary saturated subset of $E^0$ that contains all the cycles of $E$. Order the vertices and the cycles by the preorder defined in Remark~\ref{rmk:Leavitt-path-algebras} (\ref{rmk:Leavitt-path-algebras:preorder}). Then we have the following three sets of equivalent conditions:
 {
 \begin{itemize}
  \item The following are equivalent:
        \renewcommand{\theenumi}{\Alph{enumi}1}
        \begin{enumerate}
         \item\label{thm:amenLPAs:A1} $L_\K(E)$ is not algebraically amenable.
         \item\label{thm:amenLPAs:B1} $E^0$ is finite, $E^0 \setminus H = \varnothing$, and every maximal cycle is non-exclusive. 
         \item\label{thm:amenLPAs:C1} $L_\K(E)$ is unital and properly infinite
        \end{enumerate}
  \item The following are equivalent:
        \renewcommand{\theenumi}{\Alph{enumi}2}
        \begin{enumerate}
         \item\label{thm:amenLPAs:A2} $L_\K(E)$ is algebraically amenable but not properly algebraically amenable.
         \item\label{thm:amenLPAs:B2} $E^0$ is finite, $E$ is not acyclic, $E^0 \setminus H$ consists of a nonzero number of finite emitters, and every maximal cycle is non-exclusive. 
         \item\label{thm:amenLPAs:C2} $L_\K(E) = L_\K(E') \oplus L_\K(E'') $ for some directed graphs $E'$ and $E''$ such that $L_\K(E') $ has nonzero finite dimension and $L_\K(E'') $ is not algebraically amenable.
        \end{enumerate}
  \item The condition 
        \renewcommand{\theenumi}{\Alph{enumi}3}
        \begin{enumerate}
         \item\label{thm:amenLPAs:A3} $L_\K(E)$ is properly algebraically amenable
        \end{enumerate}
        holds if and only if one or more of the following conditions hold:
        \renewcommand{\theenumi}{B3\alph{enumi}} 
        \begin{enumerate}
	     \item\label{thm:amenLPAs:B3:acyclic} $E$ is acyclic;
         \item\label{thm:amenLPAs:B3:infinite} $E^0$ is infinite;
         \item\label{thm:amenLPAs:B3:emitter} $E^0 \setminus H$ contains at least one infinite emitter;
         \item\label{thm:amenLPAs:B3:exclusive} $E$ has an exclusive maximal cycle. 
        \end{enumerate}
 \end{itemize}
 }

 \end{theorem}

\begin{proof}
 Write (B3) for the inclusive disjunction \eqref{thm:amenLPAs:B3:acyclic}$\vee$\eqref{thm:amenLPAs:B3:infinite}$\vee$\eqref{thm:amenLPAs:B3:emitter}$\vee$\eqref{thm:amenLPAs:B3:exclusive}. We first observe that it suffices to show \eqref{thm:amenLPAs:B1} $\Rightarrow$ \eqref{thm:amenLPAs:C1}, \eqref{thm:amenLPAs:B2} $\Rightarrow$ \eqref{thm:amenLPAs:C2}, and (B3) $\Rightarrow$ \eqref{thm:amenLPAs:A3}. Indeed, by Corollary~\ref{cor:Properly-infinite}, we have \eqref{thm:amenLPAs:C1} $\Rightarrow$ \eqref{thm:amenLPAs:A1}, while by Example~\ref{eg:alg-amen-left-ideal} and Proposition~\ref{pro:quotient-proper}, we have \eqref{thm:amenLPAs:C2} $\Rightarrow$ \eqref{thm:amenLPAs:A2}. Notice that the three conditions \eqref{thm:amenLPAs:A1}, \eqref{thm:amenLPAs:A2} and \eqref{thm:amenLPAs:A3} are mutually exclusive, while the three conditions \eqref{thm:amenLPAs:B1}, \eqref{thm:amenLPAs:B2} and (B3) exhaust all possible situations. It thus follows from basic logic that the three converse implications also hold, i.e., we have the full cycles 
 \begin{itemize}
  \item \eqref{thm:amenLPAs:B1} $\Rightarrow$ \eqref{thm:amenLPAs:C1} $\Rightarrow$ \eqref{thm:amenLPAs:A1} $\Rightarrow$ \eqref{thm:amenLPAs:B1}, 
  \item \eqref{thm:amenLPAs:B2} $\Rightarrow$ \eqref{thm:amenLPAs:C2} $\Rightarrow$ \eqref{thm:amenLPAs:A2} $\Rightarrow$ \eqref{thm:amenLPAs:B2}, and 
  \item (B3) $\Rightarrow$ \eqref{thm:amenLPAs:A3} $\Rightarrow$ (B3).
 \end{itemize}
 We proceed now with the proofs of the three essential implications we need.
 
 \eqref{thm:amenLPAs:B1} $\Rightarrow$ \eqref{thm:amenLPAs:C1}: The unitality of $L_\K(E)$ follows directly from the finiteness of $E^0$ by Remark~\ref{rmk:Leavitt-path-algebras}(\ref{rmk:Leavitt-path-algebras:unital}). Now let $[c_1], \ldots , [c_n]$ be the maximal elements 
 of $C/{\sim}$, and pick a vertex $v_i$ in each cycle $c_i$. Since each $c_i$ is non-exclusive, by Lemma~\ref{vInfiniteIdempotent}, each $v_i$ is a properly infinite idempotent, that is, $v_i\oplus v_i \lesssim v_i $. Since $\1= \sum _{v\in E^0} v$, to show that $\1$ is properly infinite, it suffices to check that $v\lesssim p:= \sum_{i=1}^n v_i$ for all $v\in E^0$. Set $X= \{v_1,\ldots , v_n \}$. By our assumption, $E^0 = H = \ol{X}$ and $E^0$ is finite; thus there is some $k$ such that $E^0 = \Lambda _k (T(X))$. We show by induction on $r \in \N_0$ that $v\lesssim p$ for all $v\in \Lambda _r (T(X))$. For $r=0$, we have that $v\in T(X)$ and thus $v\le v_i$ for some $i$, which implies that $v\lesssim v_i \le p$. If $v\in \Lambda_r(T(X))\setminus \Lambda_{r-1} (T(X))$, then $v$ is a regular vertex and, for any $e\in s^{-1}(v)$, we have $r(e)\in \Lambda _{r-1}(T(X))$, and thus $r(e) \lesssim p$ by the induction hypothesis. Hence 
  $$v= \sum_{e\in s^{-1} (v)} ee^* \sim \bigoplus _{e\in s^{-1}(v)} r(e) \lesssim    p ^{\oplus |s^{-1}(v)|} \lesssim p,$$
 because $p$ is properly infinite. This shows that $v\lesssim p$ for all $v\in  \Lambda_r(T(X))$, completing the induction step. Therefore $\1 \oplus \1 \lesssim \1$, i.e., $L_\K(E)$ is properly infinite.
 
 \eqref{thm:amenLPAs:B2} $\Rightarrow$ \eqref{thm:amenLPAs:C2}: Define $E' = E/H$ and $E'' = M(H)$ (cf., Remark~\ref{rmk:Leavitt-path-algebras}(\ref{rmk:Leavitt-path-algebras:ideals}) 
 and (\ref{rmk:Leavitt-path-algebras:full-subgraph})). It follows from the assumptions that $E'$ has finitely many vertices and edges while $B_H= \varnothing$. By our notation 
 in Remark~\ref{rmk:Leavitt-path-algebras}(\ref{rmk:Leavitt-path-algebras:ideals}), $I(H,\varnothing)$ denotes the ideal of $L_\K(E)$ generated by $\{v \colon v \in H\}$.
 We claim that there is an isomorphism $L_\K(E'') \cong I(H,\varnothing)$. To see this, for each $v \in E^0$, we let
$\mathcal{P}_{\mathrm{min}}(v,H)$ be the set of minimal finite paths from $v$ into $H$, i.e.,
 \[
  \mathcal{P}_{\mathrm{min}}(v,H) = \{ \text{path~} \mu = e_1 \cdots e_n \colon s(e_1) = v , ~ r(e_n) \in H, ~ s(e_k) \notin H \text{~for~} k = 1, \ldots n \} \; .
 \]
 By convention, if $v \in H$, then $\mathcal{P}_{\mathrm{min}}(v,H) = \{ v \}$. Note that $\mathcal{P}_{\mathrm{min}}(v,H)$ is non-empty precisely when $v \in M(H)^0$. Since each vertex in $ E^0 \setminus H$ is regular, there are only finitely many edges that may appear in the paths in $\mathcal{P}_{\mathrm{min}}(v,H)$ for any $v \in E^0$. By minimality, these paths cannot contain cycles; thus the set $\mathcal{P}_{\mathrm{min}}(v,H)$ is finite for each $v \in E^0$. Also note that for any two different paths $\mu, \nu \in \mathcal{P}_{\mathrm{min}}(v,H)$, we have $\mu^* \nu = 0 $ in $L_\K(E)$. Thus we may define, for any $v \in E^0$, an idempotent 
 \[
  \widehat{v} = \sum_{\mu \in \mathcal{P}_{\mathrm{min}}(v,H)} \mu \mu^* \in I(H,\varnothing) \; .
 \]
 We may readily verify by Definition~\ref{def:Leavitt-path-alg} that the prescription
 \[
  v \mapsto \widehat{v} \text{~for~} v \in (E'')^0 \text{~and~} e \mapsto \widehat{s(e)} \, e \, \widehat{r(e)} \text{~for~} e \in (E'')^1 
 \]
 defines a (non-unital) graded homomorphism $L_\K(E'') \hookrightarrow L_\K(E)$ with image in $I(H,\varnothing)$. 
 This map is injective by \cite[Theorem 4.8]{tomforde}.
 On the other hand, by \cite[Lemma~5.6]{tomforde}, we have
 \begin{align*}
  I(H,\varnothing) = &~ \mathrm{span}(\{ \mu \nu^* \colon \mu \text{~and~} \nu \text{~are~paths~with~}  r(\mu) = r(\nu) \in H \}) \\
  = &~ \mathrm{span} \left( \left\{ \big( \widehat{s(\mu)} \cdot \mu \cdot \widehat{r(\mu)} ) ( \widehat{r(\nu)} \cdot \nu^* \cdot \widehat{s(\nu)} \big) \colon  r(\mu) = r(\nu) \in H \right\} \right) \; ,
 \end{align*}
 which shows that the image of the above embedding contains $I(H,\varnothing)$. Therefore we have an isomorphism $L_\K(E'') \cong I(H,\varnothing)$. 
 (We point out that another way of realizing $I(H,\varnothing )$ as a Leavitt path algebra is by using the hedgehog graph, cf. \cite[Definitions 2.5.16 and 2.5.20]{AAS}.) 
 Since $(E'')^0$ is finite, we see that $I(H,\varnothing)$ is unital as an algebra, with unit $p=\sum _{v\in M(H)^0} \widehat{v}$. It follows that $p$ is a central idempotent 
 in $L_\K (E)$, and that 
 $$ L_\K (E') = L_\K (E/H) \cong L_\K (E)/I(H,\varnothing) = (1-p)L_\K(E), $$ 
 and thus
 \[
  L_\K(E) \cong L_\K(E/H) \oplus I(H,\varnothing) \cong L_\K(E')  \oplus L_\K(E'') \; .
 \]
 Since $E/H$ is a finite graph with no cycle, by Remark~\ref{rmk:Leavitt-path-algebras}(\ref{rmk:Leavitt-path-algebras:finite-acyclic}), we see that $L_\K(E')$ has finite dimension. 
 On the other hand, by our construction of the graph $E''$, it inherits all the maximal cycles of $E$, which are all non-exclusive, and $(E'')^0$ is equal to the smallest hereditary 
 saturated subset (with respect to $E''$) containing all the cycles. Thus $E''$ satisfies \eqref{thm:amenLPAs:B1}. Since we have already proved 
 \eqref{thm:amenLPAs:B1} $\Rightarrow$ \eqref{thm:amenLPAs:C1} $\Rightarrow$ \eqref{thm:amenLPAs:A1}, we conclude that $L_\K(E'')$ is not algebraically amenable. 

 \eqref{thm:amenLPAs:B3:acyclic}$\vee$\eqref{thm:amenLPAs:B3:infinite}$\vee$\eqref{thm:amenLPAs:B3:emitter}$\vee$\eqref{thm:amenLPAs:B3:exclusive} $\Rightarrow$ \eqref{thm:amenLPAs:A3}: We first 
 observe that when \eqref{thm:amenLPAs:B3:acyclic} holds and \eqref{thm:amenLPAs:B3:infinite} fails, i.e., when $E$ is finite and acyclic, Remark~\ref{rmk:Leavitt-path-algebras} (\ref{rmk:Leavitt-path-algebras:finite-acyclic}) 
 tells us that $L_\K(E)$ is finite dimensional and thus properly algebraically amenable. 
 
 Apart from this easy case, $L_\K(E)$ is always infinite-dimensional, so by Proposition~\ref{pro:proper-alg-amen2}, it suffices to show that, given any $\varepsilon >0$, any $N \in \N$, and any finite subset $\mathcal F$ of $L_\K(E)$, we can find an $(\mathcal{F}, \varepsilon)$-F\o lner subspace $W$ in $L_\K(E)$ with $\dim(W) \geq N$. Since each element of $L_\K(E)$ is a linear combination of terms of the form $\lambda \rho ^*$, where $\lambda $ and $\rho$ are paths such that $r(\lambda )= r(\rho )$, without loss of generality we can assume that $\mathcal F$ consists of elements of this form, say $\mathcal F = \{ \lambda_1\rho_1^*,
 \ldots , \lambda _r \rho_r^* \}$. 
 
 First, we assume \eqref{thm:amenLPAs:B3:infinite} holds, i.e., $E^0$ is infinite. Then we can find a subset $X \subset E^0$ with $|X | = N$ and $X \cap \{ s(\rho_1) , \ldots, s(\rho_r) \} = \varnothing$. Put $W = \mathrm{span}(X)$. It then follows that $\lambda _j \rho_j^* W = 0 $ for $j = 1, \ldots, r$. Hence $W$ is an $(\mathcal{F}, 0)$-F\o lner subspace with $\dim(W) \geq N$.
 
 Next, we assume \eqref{thm:amenLPAs:B3:emitter} holds but \eqref{thm:amenLPAs:B3:infinite} fails, i.e. $E^0$ is finite and $E^0\setminus H$ contains at least one
infinite emitter. Let $v$ be a maximal element among all infinite emitters of $E^0\setminus H$. Then $M(v)$ contains no cycle and includes only
finitely many vertices with no infinite emitter, and thus it also has only finitely many edges. By
Remark~\ref{rmk:Leavitt-path-algebras}(\ref{rmk:Leavitt-path-algebras:finite-acyclic}), there are only finitely many paths in $E$ ending in $v$. 
 Since $s_E^{-1} (v)$ is infinite, there is $Y \subset s_E^{-1} (v)$ such that $|Y| = N$ and any $e \in Y$ is not contained in any of the paths $\rho_i$, for $i = 1, \ldots, r$. Define $W$ to be the linear span of the finite set
 \[
  \{ \tau e \in L_\K(E)  \colon \tau \text{~is~a~path~ending~in~} v , ~ e \in Y \} \; .
 \]
 Notice that $\dim(W) \geq |Y| = N$. We claim that $\lambda_i \rho_i^* W\subset W$ for $i = 1, \ldots, r$. Indeed, since $e$ is not an edge in $\rho_i$, the only way that the product $(\lambda_i \rho_i^*)(\tau e )$ is nonzero is that $\tau = \rho_i \tau'$ for some path $\tau'$ ending in $v$, whence
 $$  (\lambda_i \rho_i^*)(\tau e ) = \lambda_i\tau ' e  \in W .$$
 This shows our claim. Hence $W$ is an $(\mathcal{F}, 0)$-F\o lner subspace with $\dim(W) \geq N$.
 
 Finally, we assume \eqref{thm:amenLPAs:B3:exclusive} holds but both \eqref{thm:amenLPAs:B3:infinite} and \eqref{thm:amenLPAs:B3:emitter} fail, i.e., $E^0$ is finite, $E^0 \setminus H$ consists of regular vertices, and there is an exclusive maximal cycle, which we denote by $c$. Let $v_0$ be a vertex in $c$ and let $\mu_0$ be the representative of $c$ based at $v_0$. The subgraph $M(v_0)$ of $E$ has the unique cycle $c$, and every vertex in $M(v_0)$ connects to it via paths. We claim that every vertex $v \in M(v_0)^0$ is regular in $M(v_0)$. 
 Indeed, by Remark~\ref{rmk:Leavitt-path-algebras}(\ref{rmk:Leavitt-path-algebras:full-subgraph}), every vertex in $c$ only emits one edge in $M(v_0)$. On the other hand, any $v \in M(v_0)^0 \setminus H$ 
 is regular even in $E$ by our assumption. It remains to show that any $v \in M(v_0)^0 \cap H \setminus c^0$ is regular. For this we let $X \subset H$ consist of all the vertices of maximal cycles of $E$. 
 Then by Remark~\ref{rmk:Leavitt-path-algebras}(\ref{rmk:Leavitt-path-algebras:hereditary-saturated}), $H = \ol{X} = \bigcup_{k=0}^\infty \Lambda_k(T(X))$. It is clear by the maximality of the cycles that $M(v_0)^0 \cap T(X) = c^0$. Hence for any $v \in M(v_0)^0 \cap H \setminus c^0$, there is some $k \in \N_0$ such that $v \in \Lambda_{k+1}(T(X)) \setminus \Lambda_{k}(T(X)) $; thus $v$ is a regular vertex (even in $E$) by the definition of $ \Lambda_{k+1}(T(X))$. This proves the claim. Now for each $v \in E^0$, we let $\mathcal{P}_{\mathrm{min}}(v,v_0)$ be the set of minimal finite paths from $v$ to $v_0$, i.e.,
 \[
  \mathcal{P}_{\mathrm{min}}(v,v_0) = \{ \text{path~} \mu = e_1 \cdots e_n \colon s(e_1) = v , ~ r(e_n) = v_0, ~ s(e_k) \not= v_0 \text{~for~} k = 1, \ldots n \} \; ,
 \]
 By convention, $\mathcal{P}_{\mathrm{min}}(v_0,v_0) = \{ v_0 \}$. Note that $\mathcal{P}_{\mathrm{min}}(v,v_0)$ is a subset of all paths in $M(v_0)$ for each $v \in E^0$ and is non-empty 
 precisely when $v \in M(v_0)^0$. Since every vertex $v \in M(v_0)^0$ is regular in $M(v_0)$, there are only finitely many edges that may appear in the paths 
 in $\mathcal{P}_{\mathrm{min}}(v,v_0)$ for any $v \in E^0$. By minimality, these paths cannot contain cycles; thus the set $\mathcal{P}_{\mathrm{min}}(v,v_0)$ is finite for 
 each $v \in E^0$. Thus the union $\mathcal{P} = \bigcup_{v \in E^0} \mathcal{P}_{\mathrm{min}}(v,v_0) $ of all minimal paths ending in $v_0$ is also finite. Note that any path ending in $v_0$ can be written 
 uniquely as $\gamma \mu_0^k$ for some $\gamma \in \mathcal{P}$ and $k \in \N_0$. For each $k \in \N_0$, define a linear subspace $W_k$ of $L_\K(E)$ by
 \[
  W_k = \mathrm{span}( \{ \gamma \mu_0^k \colon \gamma \in \mathcal{P}  \} ) 
 \]
 Thus for any different $k,l \in \N_0$, we have $\dim(W_k) = |\mathcal{P}|$ and the collection of subspaces $\{ W_k \}$ is independent.  
 Let $N_1 \in \N$ be such that $N_1 |\mu_0|$ is greater than the length of each path among $\lambda _1, \ldots, \lambda_r, \rho_1, \ldots, \rho_r $, where $|\mu_0|$ is the 
 length of $\mu_0$. For any $j \in \{1, \ldots, r\}$, $\gamma \in \mathcal{P}$ and $k \in \N$ with $k \geq N_1$, we claim that 
 \[
  \lambda_j\rho_j^* \gamma \mu_0^k  \in \sum_{l = k-N_1}^{k+N_1} W_l \; .
 \]
 Indeed, this is trivial when $\rho_j^* \gamma \mu_0^k = 0$. If $\rho_j^* \gamma \mu_0^k \not= 0$, since $|\gamma \mu_0^k| > |\rho_j|$, we have $\gamma \mu_0^k = \rho_j \tau$ for some path $\tau$ ending in $v_0$. Hence $\lambda_j\rho_j^* \gamma \mu_0^k = \lambda_j \tau = \theta \mu_0^l$ for some $\theta \in \mathcal{P}$ and $l \in \N$. If $|\gamma | > |\rho_j|$, then $s(\tau) \notin c^0$ and thus $l = k$. Otherwise we have the estimates
 \[
  k |\mu_0| - |\rho_j| \leq l |\mu_0| \leq k |\mu_0| + |\lambda_j| \;.
 \]
 In either case, we have $k-N_1 \leq l \leq k+N_1$. This proves the claim. Now let $N_2 \in \N_0$ be such that $N_2 > N + N_1$ and $\frac{2 N_1}{N_2-N_1} \le \varepsilon $, and define 
 \[
  W = \sum_{k=N_1 + 1}^{N_2} W_k \; .
 \]
 Then $\dim(W) = |\mathcal{P}| (N_2 - N_1) \geq N$ and for any $j \in \{1, \ldots, r\}$, we have
 \[
  \frac{\dim (\lambda_j\rho_j^* W + W)}{\dim (W)} \leq \frac{ \dim( \sum_{k=1}^{N_2+N_1} W_k ) }{ \dim( \sum_{k=N_1+1}^{N_2} W_k ) } = \frac{| \mathcal P | (N_2+ N_1) }{| \mathcal P | (N_2-N_1)} \le 1 +\varepsilon \; .
 \]
 Hence $W$ is an $(\mathcal{F}, \varepsilon )$-F\o lner subspace with $\dim(W) \geq N$.
 
 Therefore any of the conditions \eqref{thm:amenLPAs:B3:acyclic}, \eqref{thm:amenLPAs:B3:infinite}, \eqref{thm:amenLPAs:B3:emitter} and \eqref{thm:amenLPAs:B3:exclusive} implies that $L_\K(E)$ is properly algebraically amenable. 
\end{proof}

We highlight the following trivial consequence of Theorem~\ref{thm:amenLPAs}:

\begin{corollary}
 \label{cor:equiv-notamenable-prop-inf}
 Let $E$ be a graph with finitely many vertices and let $\K$ be a field. Then the (unital) Leavitt path algebra $L_\K (E)$ is not algebraically 
 amenable if and only if it is properly infinite. 
 \end{corollary}

 \begin{remark}
  \label{rem:growth-conditions}
  It is well-known (\cite[Proposition 3.1]{Elek03}) that a finitely generated $\K$-algebra of subexponential growth is amenable. On the other hand, it has been shown in \cite{AAJZ}
  that, for a finite graph $E$, the Leavitt path algebra $L_{\K}(E)$ either has exponential growth or has polynomially bounded growth. Moreover, by \cite[Theorem 5 (1)]{AAJZ}, $L_{\K}(E)$ has 
 polynomially bounded growth if and only if every cycle of $E$ is an exclusive cycle, and in this case a precise formula for the Gelfand-Kirillov dimension of $L_{\K}(E)$ is obtained (\cite[Theorem 5 (2)]{AAJZ}).
 Comparing this with Theorem~\ref{thm:amenLPAs}, we see that there are finite graphs such that $L_{\K}(E)$ is algebraically amenable and has exponential growth (just consider the graph $E$ of Example~\ref{exam:pathalgebra2}). 
   \end{remark}

 Since $L_\mathbb{K}(E)$ admits an involution (see for instance \cite{tomforde}), left and right amenability is equivalent for these algebras.
 Moreover the above proof shows that we can ``localize'' amenability in certain parts of the graph (in analogy with the metric space situation,
 cf., Subsection~\ref{subsec:generalized-metric}). 
 We provide a simple example that shows that the situation is quite different
 when we consider the usual path algebras. 
 
 \begin{definition}\label{def:path-alg}
 Given an arbitrary graph $E$ and a field $\mathbb{K}$, the \emph{path $\mathbb{K}$-algebra }$\mathbb{K} E$ is defined to be the $\mathbb{K}$-algebra generated by a set $\{v:v\in E^{0}\}$ of pairwise orthogonal idempotents together with a set of variables $\{e:e\in E^{1}\}$ which satisfy $s(e)e=e=er(e)$ for all $e\in E^{1}$.
\end{definition}
 
 In other words, the path algebra is linearly spanned by all paths in $E$, with the multiplication given by concatenation of paths (or zero if two paths cannot be concatenated). 
 
 \begin{example}
  \label{exam:pathalgebra}
  Let $E$ be the following graph:
 \[
{
\def\labelstyle{\displaystyle}
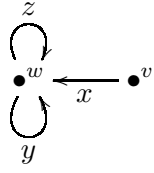
\xymatrix{ \bullet^w  \uloopr{z}\dloopr{y}
 & \bullet^v \ar[l]^{x} }} 
\]
Let $\cA$ be the corresponding path algebra $\K E$. We claim that $\cA$ is left properly algebraically amenable but not right algebraically amenable. 

To this end, we first observe, by checking on all paths in $E$, that for any $a \in \cA$, we have $a v = vav = \kappa v$ for some $\kappa \in \K$, while $wa = waw$ and $vaw = xbw$ for some $b \in w\cA w $. Since $v + w = \1$, we have the linear decomposition 
\[
 \cA = w \cA w \oplus v \cA w \oplus v \cA v = w \cA w \oplus x \cA w \oplus \K v \; .
\]
Define the following linear maps:
\begin{align*}
 \lambda \colon &~ \cA \to w \cA = w \cA w , && a \mapsto wa \; ; \\
 \rho \colon &~ \cA \to \cA w , && a \mapsto aw \; ; \\
 \phi \colon &~ w \cA w \to x \cA w , && a \mapsto x a \; .
\end{align*}
Then $\lambda$ and $\rho$ are surjections with kernels $v\cA$ and $\cA v$ ($= \K v$), respectively, while $\phi$ is a bijection. Also observe that the subalgebra $w \cA w$ is isomorphic to the free algebra on two generators, and hence not algebraically amenable as it cannot carry an invariant dimension measure. In particular, both $w \cA w$ and $x \cA w$ have countably infinite dimension.

To see that $\cA$ is left properly algebraically amenable, we choose an arbitrarily large finite-dimensional subspace $W$ of $ x \cA w$ and note that $\cA W = \cA (v W ) = \K v W = W$, i.e., $W$ is an $(\cA, 0)$-F\o lner subspace. 

It remains to show that $\cA$ is not right algebraically amenable. Since $w \cA w$ is not algebraically amenable, by Lemma~\ref{lem:local-doubling}, there exists a finite subset $\cF_0 \subset w \cA w$ such that for any finite-dimensional subspace $W \subset w \cA w$, we have $\dim (W \mathcal F _0 + W  )\ge 3 \dim (W)$. Without loss of generality, we may assume $w \in \cF_0$. Now define 
\[
 \cF = \cF_0 \cup \{ x, v \} \; .
\]
Given an arbitrary nontrivial finite-dimensional subspace $W \subset \cA$, we would like to show that $\dim (W \mathcal F  + W  )\ge 2 \dim (W)$. 

First, if $W = \K v$, then $W \mathcal F  + W  = \K x \oplus \K v$, which has dimension $2$, as desired. Now if $W \not= \K v$, or equivalently, $W w \not= 0$, then notice that
\begin{align*}
 \dim(W) = &~ \dim(\rho(W)) + \dim(\mathrm{ker}(\rho) \cap W) \\
 = &~ \dim(W w) + \dim( \K v \cap W) \\
 = &~ \dim(\lambda(W w)) + \dim(\mathrm{ker}(\lambda) \cap W w) + \dim( \K v \cap W) \\
 = &~ \dim(w W w) + \dim( v \cA  \cap W w) + \dim( \K v \cap W) \\
 = &~ \dim(w W w) + \dim( x \cA w \cap W w) + \dim( \K v \cap W) \; .
\end{align*}
Similarly, we have
\begin{align*}
 \dim(W \cF + W) = &~ \dim ( W \cF_0 + Wx + W v + W) \\
 = &~ \dim (w ( W \cF_0 + Wx + W v + W ) w ) \\
  & +   \dim( x \cA w \cap ( W \cF_0 + Wx + W v + W) w) \\
  & + \dim( \K v \cap ( W \cF_0 + Wx + W v + W)) \\
 = &~ \dim (w W w \cF_0 ) + \dim( x \cA w \cap ( W w \cF_0 + W x ) ) + \dim( v W v) \\
 \geq &~ \dim (w W w \cF_0 ) + \dim( ( x \cA w \cap  W w ) \cF_0  ) + \dim( v W v ) \\
 = &~ \dim (w W w \cF_0 ) + \dim( \phi^{-1} ( x \cA w \cap  W w ) \cF_0 ) + \dim( v W v) \\
 \geq &~ 3 \dim(w W w) + 3 \dim(\phi^{-1} ( x \cA w \cap  W w )) + \dim( v W v) \\
 = &~ 3 \dim(w W w) + 3 \dim( x \cA w \cap  W w ) + \dim( v W v) \\
 = &~ 3 \dim(W w) + \dim(v W v) \; .
\end{align*}
Here we used the fact that $\phi$ is a bijection and preserves multiplication from the right. Depending on whether $v \in W$ and whether $Wv = 0$, the pair $(\dim(v W v) , \dim( \K v \cap W) )$ may take value among $(0,0)$, $(1, 0)$ and $(1,1)$. In any case, since $\dim(W w) \geq 1$ by our assumption, we have
\[
 \frac{\dim(W \cF + W)}{ \dim(W) } \geq \frac{ 3 \dim(W w) + \dim(v W v) }{ \dim(W w) + \dim( \K v \cap W) } \geq \frac{3 \dim(W w)+1}{ \dim(W w) + 1} \ge  2
\]
as desired. Therefore $\cA$ is not algebraically amenable. \qed
 \end{example}

 The next example is similar to the above. It shows that having a maximal exclusive cycle is not enough to guarantee the (right) amenability of path algebras 
 (compare with Theorem~\ref{thm:amenLPAs}).

  \begin{example}
  \label{exam:pathalgebra2}
  Let $E$ be the following graph:
 \[
{
\def\labelstyle{\displaystyle}
\xymatrix{ \bullet^w  \uloopr{z}\dloopr{y}
 & \bullet^v \ar[l]^{x} \uloopd{t}}}
\]
Here we also have that the path algebra $\mathcal A : = \mathbb{K}E$ is left properly algebraically amenable but not right algebraically amenable, despite the existence of an exclusive maximal cycle. Since the proof is similar to the one in the previous example, we only give a sketch, leaving the details to the reader.

 In this case, we have 
 a linear decomposition 
 $$\mathcal A = w\mathcal A w \oplus v\mathcal A v \oplus v\mathcal A w \cong w\mathcal A w \oplus \mathbb{K}[t] v \oplus x\mathcal A w \oplus tx\mathcal A w \oplus t^2 x\mathcal A w \oplus \cdots \; .$$ 
 For the left algebraic amenability, we can use a proper F\o lner net inside $\mathbb{K}[t]v$. On the other hand, for the right algebraic non-amenability, we again take $\cF_0 \subset w \cA w$ as in the previous example and set $\cF = \cF_0 \cup \{ x, v \} $. Given an arbitrary finite-dimensional subspace $W \subset \cA$, if $\dim(\cA v \cap W) \geq \frac{3}{5} \dim(W)$, then
 \[
  \dim(W \cF) \geq \dim( (\cA v \cap W ) \cdot \{x,v\} ) = 2 \dim(\cA v \cap W ) \geq \frac{6}{5} \dim(W) \; .
 \]
 Otherwise, we have $\dim(W w) = \dim(W / (\cA v \cap W) ) = \dim(W) - \dim(\cA v \cap W) > \frac{2}{5} \dim(W)$. Note that $W w$ is contained in a finitely generated free right $w\mathcal A w$-module $w\mathcal A w \oplus x\mathcal A w \oplus tx\mathcal A w \oplus t^2 x\mathcal Aw \oplus \cdots \oplus t^k x\mathcal Aw $ for some $k \in \N_0$. Thus by iterating the argument we used in the previous example (where we had $Ww \subset w\mathcal A w \oplus x\mathcal A w$), we can show 
 \[
  \dim(W \cF) \geq \dim(  W w  \cdot \cF_0 ) \geq 3 \dim( W w ) > \frac{6}{5} \dim(W) \; .
 \]
 Thus $\cA$ is not right algebraically amenable. \qed 
 \end{example}

\section{Translation algebras on coarse spaces}\label{sec:alg-amenable2}

To conclude we will illustrate the close relation between amenability for metric spaces and algebraic amenability for $\K$-algebras, in view of the natural bridge 
between the two settings \textendash~the construction of translation algebras (see, e.g., \cite[Chapter~4]{Roe03}). Let us recall this construction. 

Let $(X,d)$ be a locally finite extended metric space as in Section~\ref{sec:amenability-metric} and $\K$ an arbitrary field. We denote by $\K[X]$ the $\K$-linear space generated by the basis $X$, and by $\mathrm{End}_{\K}(\K[X])$ the algebra of $\K$-linear endomorphism of $\K[X]$. For the sake of clarity, we denote by $\delta_x$ the basis element of $\K[X]$ corresponding to a point $x \in X$. We also sometimes think of an element $T \in \mathrm{End}_{\K}(\K[X])$ as a matrix indexed by $X$, and define $T_{xy} \in \K$ as its entry at $(x,y ) \in X \times X$, so that $T(\delta_y) = \sum_{x\in X} T_{xy} \delta_x $ for any $y \in X$. 

For any partial translation $t$ on $X$ (cf.~Definition~\ref{def:part-transl}), we define $V_t \in \mathrm{End}_{\K}(\K[X])$ by
\begin{equation}\label{eq:partial-trans}
 V_t(\delta_x):=
\begin{cases}
           \delta_{t(x)} & \text{if}\  x \in \mathrm{dom}(t)\\
           0 & \text{if}\  x \notin \mathrm{dom}(t) \;.
\end{cases}
\end{equation}

Note that for any two partial translations $t$ and $t'$ on $X$, we have $V_t V_{t'} = V_{t \circ t'}$. In other words, $t \mapsto V_t$ gives a representation of the semigroup $\mathrm{PT}(X)$. 

\begin{definition}\label{def:translation-algebra}
 The \emph{translation $\K$-algebra} $\K_\mathrm{u}(X)$ is the (unital) $\K$-subalgebra of $\mathrm{End}_{\K}(\K[X])$ generated by $V_t$ for all the partial translations $t$ on $X$. 
\end{definition}

Any subset $A \subset X$ gives rise to an idempotent $V_{\mathrm{Id}_A}$ in $\K_\mathrm{u}(X)$, where $\mathrm{Id}_A$ is the identity map on $A$. For the sake of simplicity, we denote this idempotent by $P_A$. In particular, $P_X$ is equal to the unit of $\mathrm{End}_{\K}(\K[X])$. Note that we have the identities
\[
 V_{t^{-1}} V_t=P_{\mathrm{dom}(t)} \qquad \mathrm{and} \qquad V_t V_{t^{-1}}=P_{\mathrm{ran}(t)} 
\]
for any partial translation $t$ on $X$. Moreover, any element in $\K_\mathrm{u}(X)$ can be linearly spanned by the generators $V_t$.

Given a matrix $T \in \mathrm{End}_{\K}(\K[X])$ it is useful to consider its propagation as defined by
\begin{align*}
p(T):=\sup\Big\{d(x,y):x,y\in X\quad\mathrm{and}\quad T_{xy}\neq 0\Big\}.
\end{align*}
It is clear that every element in the translation $\K$-algebra has finite propagation and that for any $A\subset X$ we have $p(P_A)=0$.

\begin{remark}\label{rem:Roe-direct-sum}
 One can easily see that whenever we have a decomposition of an extended metric space $X$ into a finite disjoint union $X_1 \sqcup \ldots \sqcup X_n$ with infinite distance between each pair of subspaces, then the associated idempotents $P_{X_1} , \ldots, P_{X_n}$ are central and mutually orthogonal, and add up to the unit, which induces a direct sum decomposition
 \[
  \K_\mathrm{u}(X) \cong \bigoplus_{i = 1} ^n \K_\mathrm{u}(X_i) \; .
 \]
\end{remark}

\begin{theorem}\label{theorem:main1}
Let $(X,d)$ be a locally finite extended metric space and let $\K_\mathrm{u}(X)$ be its translation $\K$-algebra. Let $n \geq 2$ be a natural number. 
Then the following conditions are equivalent: 
\begin{enumerate}
 \item \label{theorem:main1:coarse} $(X,d)$ is amenable.
 \item \label{theorem:main1:algebraic} $\K_\mathrm{u}(X)$ is algebraically amenable.
 \item \label{theorem:main1:properly-infinite} $\K_\mathrm{u}(X)$ is not properly infinite.
 \item \label{theorem:main1:Leavitt-A} $\K_\mathrm{u}(X)$ does not contain the Leavitt algebra $L_{\K}(1,n)$ as a unital $\K$-subalgebra. 
\end{enumerate}
\end{theorem}

\begin{proof} 
(\ref{theorem:main1:coarse}) $\Rightarrow$ (\ref{theorem:main1:algebraic}): 
Consider $\varepsilon > 0$ and a finite set $\mathcal{F}\subset \K_\mathrm{u}(X)$. We may assume that any element in $\cF$ has propagation at most $R>0$.
Since $(X,d)$ is amenable, and using the conventions in Definition~\ref{def:metric-amenability},
there exists a (finite, non-empty set) $F\in\mathrm{\mbox{F\o l}}(R, \varepsilon)$. We first show that we may assume that $F$ is contained in a single coarse component of $X$.
Indeed, write $F= \bigsqcup_{i=1}^N F_i$, where $F_i$, $i=1,\dots ,N$, are the coarse components of $F$ (i.e, the non-empty intersections of $F$ with the different coarse components of $X$).
We then have $\sum_{i=1}^N |\partial_R(F_i)| / |F| < \varepsilon $. Suppose that  $ |\partial_R(F_i)| / |F_i| > \varepsilon$ for all $i$. Then we have
$$\sum_{i=1}^N \frac{|\partial_R(F_i)|}{|F|} = \sum_{i=1}^N \frac{|F_i|}{|F|} \cdot \frac{|\partial_R(F_i)|}{|F_i|} > \Big( \sum_{i=1}^N \frac{|F_i|}{|F|} \Big) \varepsilon = \varepsilon ,$$ 
a contradiction.   Thus, by replacing $F$ with some of the its coarse components, we may assume that $F$ is contained in a coarse component of $X$. 

It follows from the definition of propagation that whenever $d(Y, Y') > R$, then any $T\in\cF$ satisfies $P_Y T P_{Y'} = 0$. 
Now we define the following linear subspace in $\K_\mathrm{u}(X)$ (in fact a subalgebra),
\[
 W := P_F \K_\mathrm{u}(X) P_F \subset \K_\mathrm{u}(X)\;,
\]
which satisfies that $\dim W = |F|^2$, because $F$ is contained in a single coarse component of $X$. 

We analyze next for any $T\in\mathcal{F}$ the subspace $T W $ as follows. To simplify expressions
we will use the standard notation for the commutator of two operators: $[T, B] := TB-BT$. Using the 
notation of $R$-boundaries and neighborhoods of Section~\ref{sec:amenability-metric} we have
\[
 \1 = P_F + P_{X \setminus F} = (P_{N^-_{R}F} + P_{\partial^-_R F}) + (P_{\partial^+_R F} + P_{X \setminus N^+_R F} )
\]
as well as
\[
 P_{N_{R}^-F} T P_{X \setminus F} = P_{X \setminus F} T P_{N_{R}^-F} = P_{X \setminus N_R^+ F} T P_F = P_F T P_{X \setminus N_R^+ F} = 0 \;.
\]
Then we have
\begin{align*}
 T P_F =&\ (P_F + P_{\partial^+_R F} + P_{X \setminus N_R^+ F} ) T P_F \\
 =&\ P_F T P_F + P_{\partial^+_R F} T (P_{N_{R}^-F} + P_{\partial^-_R F}) + 0 \\
 =&\ P_F T P_F + 0 + P_{\partial^+_R F} T P_{\partial^-_R F} \;,
\end{align*}
and similarly
\[
 P_F T = P_F T P_F + P_{\partial^-_R F} T P_{\partial^+_R F}.
\]
Hence 
\begin{equation}\label{comm}
 [T, P_F] = P_{\partial^+_R F} T P_{\partial^-_R F} - P_{\partial^-_R F} T P_{\partial^+_R F} \; ,
\end{equation}
and
\begin{eqnarray} \nonumber
 T W   &=& \{T\, P_F B P_F \,:\, B\in\K_\mathrm{u}(X)\; \} \\ \nonumber
            &=& \{P_F TB P_F + [T,P_F]\,BP_F \, :\, B\in\K_\mathrm{u}(X)\; \} \\
            &=& \{P_F TB P_F + P_{\partial^+_R F} T P_{\partial^-_R F} \,BP_F - P_{\partial^-_R F} T P_{\partial^+_R F} \,BP_F \, :\, B\in\K_\mathrm{u}(X)\; \} \\
       &\subseteq & W  + P_{\partial^+_R F} \K_\mathrm{u}(X) P_{F} + P_{\partial^-_R F} \K_\mathrm{u}(X) P_{F} \;.
\end{eqnarray}
Therefore we have the following estimates for any $T\in\cF$:
\begin{eqnarray*} 
 \frac{\dim(TW  +W )}{\dim(W )} &\leq & \frac{\dim(W ) + \dim(P_{\partial^+_R F} \K_\mathrm{u}(X) P_{F} ) + \dim( P_{\partial^-_R F} \K_\mathrm{u}(X) P_{F})}{\dim(W )}  \\
                                &\leq &  1+ \frac{|F|\,|\partial^+_R F|+|F|\,|\partial^-_R F|}{|F|^2}  \\
                                &=    &  1+\frac{ \, |\partial_R F| }{|F|} \;\leq\; 1+ \varepsilon\;.
\end{eqnarray*}
This shows that $\K_\mathrm{u}(X)$ is algebraically amenable. 

(\ref{theorem:main1:algebraic}) $\Rightarrow$ (\ref{theorem:main1:properly-infinite}): This implication follows from 
Corollary~\ref{cor:Properly-infinite}.
  
(\ref{theorem:main1:properly-infinite}) $\Rightarrow$ (\ref{theorem:main1:Leavitt-A}): Suppose that 
for some $n\geq 2$ the Leavitt algebra $L(1,n)$ unitally embeds into $\C_\mathrm{u}(X)$. Then, any two distinct pairs
of generators $X_i,Y_i$, $X_j,Y_j$, $i\not=j$, of $L(1,n)$ implement the proper infiniteness of 
$\K_\mathrm{u}(X)$.  
   
(\ref{theorem:main1:Leavitt-A}) $\Rightarrow$ (\ref{theorem:main1:coarse}): Assume that $(X,d)$ is not amenable. Then by 
Theorem~\ref{theorem:amenable-extended-metric} $X$ is paradoxical, i.e., there is a partition $X=X_+\sqcup X_-$ and partial 
translations $t_{\pm}\colon X\to X_\pm$. The corresponding generators of the translation algebra $V_{t_\pm}$, $V_{t_\pm^{-1}}$ 
satisfy 
\[
 V_{t_+}V_{t_+^{-1}}+V_{t_-}V_{t_-^{-1}} =\1\;,\quad V_{t_\pm^{-1}}V_{t_\pm}=\1\quad\mathrm{and}\quad V_{t_\pm^{-1}}V_{t_\mp}=0\;.
\]
This shows that $L(1,2)$ unitally embeds into the translation $\K$-algebra. The result then follows from the fact that $L(1,n)$ unitally embeds into $L(1,2)$ (see \cite[Theorem 4.1]{BS16}).
\end{proof}

We also have an analogous result for proper amenability. We will use the following terminology. Given two algebras $\mathcal A$ and $\mathcal B$, we say that $\mathcal A$ is a {\it finite-dimensional extension
of} $\mathcal B$ in case there is a finite-dimensional two-sided ideal $I$ of $\mathcal A$ such that $\mathcal A/ I \cong \mathcal B$ \footnote{This is in agreement with the non-universal
convention of calling the algebra $\mathcal A$ above an extension of $\mathcal B$ by $I$.}.

\begin{theorem}\label{theorem:Roe-proper-amenability}
 Let $(X,d)$ be a locally finite extended metric space and let $\K_\mathrm{u}(X)$ be its translation $\K$-algebra.
 Let $n \geq 2$ be a natural number. 
Then the following conditions are equivalent: 
\begin{enumerate}
 \item \label{theorem:Roe-proper-amenability:coarse} $(X,d)$ is properly amenable.
 \item \label{theorem:Roe-proper-amenability:algebraic} $\K_\mathrm{u}(X)$ is properly algebraically amenable.
 \item \label{theorem:Roe-proper-amenability:properly-infinite} $\K_\mathrm{u}(X)$ is not a finite-dimensional extension of a properly infinite $\K$-algebra.
\end{enumerate}
\end{theorem}

\begin{proof} We may assume that $X$ is infinite. 

(\ref{theorem:Roe-proper-amenability:coarse}) $\Rightarrow$ (\ref{theorem:Roe-proper-amenability:algebraic}): Assume that $(X,d)$ is properly amenable and recall 
the proof of the implication (\ref{theorem:main1:coarse}) $\Rightarrow$ (\ref{theorem:main1:algebraic}) in Theorem~\ref{theorem:main1}. 
For $R > 0$, $\varepsilon >0$ and  $N\in\N$, by Lemma~\ref{lem:proper-cardinality}, we can choose 
a (finite, non-empty) set $F\in\mathrm{\mbox{F\o l}}(R, \frac{\varepsilon}{2})$ with
$|F|\geq 2 N$. Let $F= \bigsqcup_{i\in I} F_i$ be the decomposition of $F$ into its coarse components. Let
\[
	I' := \left\{ i \in I \colon \frac{|\partial_R F_i |}{|F_i|} \leq \varepsilon \right\}
\]
and let $F' := \bigsqcup_{i\in I'} F_i$. We observe that $|F'| \geq \frac{1}{2} |F| \geq N$. Indeed, if this were not true, then 
\[
	\frac{|\partial_R F |}{|F|} \geq \frac{\sum_{i \in I \setminus I'} |\partial_R F_i |}{|F|} > \frac{\sum_{i \in I \setminus I'} \varepsilon | F_i |}{|F|} =  \frac{ \varepsilon | F \setminus F' |}{|F|}  > \frac{ \varepsilon \frac{1}{2} | F |}{|F|}  = \frac{\varepsilon}{2}  \; ,
\]
a contradiction to $F\in\mathrm{\mbox{F\o l}}(R, \frac{\varepsilon}{2} )$. For each $i \in I'$, let $W_i := P_{F_i}  \K_\mathrm{u}(X) P_{F_i}$. Then as in the proof of  the implication (\ref{theorem:main1:coarse}) $\Rightarrow$ (\ref{theorem:main1:algebraic}) in Theorem~\ref{theorem:main1}, we have $\dim W_i = |F_i|^2$ and for any $T$ with propagation no more than $R$, we have $\dim (T W_i + W_i ) \leq |F_i| (|F_i| + |\partial_R F_i|) \leq |F_i|^2 (1 + \varepsilon)$. Hence if we let $W = \sum_{i \in I'} W_i$, we have 
\[
	\dim(W) = \sum_{i \in I'} \dim(W_i ) = \sum_{i \in I'} |F_i|^2 \geq \sum_{i \in I'} |F_i| = |F'| \geq N
\]
and for any $T$ with propagation no more than $R$
\[
	\frac{\dim(TW  +W )}{\dim(W )} = \frac{\sum_{i \in I'} \dim (T W_i + W_i )}{\sum_{i \in I'} \dim(W_i )} \leq \frac{\sum_{i \in I'} |F_i|^2 (1 + \varepsilon)}{\sum_{i \in I'} |F_i|^2} = 1 + \varepsilon
\]
Hence, by Proposition~\ref{pro:proper-alg-amen2}, we have that $\K_\mathrm{u}(X)$ is properly algebraically amenable.

(\ref{theorem:Roe-proper-amenability:algebraic}) $\Rightarrow$ (\ref{theorem:Roe-proper-amenability:properly-infinite}): Suppose that $\K_\mathrm{u}(X)$ is a finite-dimensional extension of 
a properly infinite $\K$-algebra, that is, there is a finite-dimensional two-sided ideal $I$ of $\K_\mathrm{u}(X)$ such that $\K_\mathrm{u}(X) / I$ is properly infinite. By Corollary~\ref{cor:Properly-infinite}, $\K_\mathrm{u}(X) / I$ is not algebraically amenable, and thus not properly algebraically amenable, either. By Proposition~\ref{pro:quotient-proper}, it follows that $\K_\mathrm{u}(X)$ is not properly algebraically amenable.  

(\ref{theorem:Roe-proper-amenability:properly-infinite}) $\Rightarrow$ (\ref{theorem:Roe-proper-amenability:coarse}): Assume that $\K_\mathrm{u}(X)$ is not a finite-dimensional extension of a properly infinite $\K$-algebra. In particular, itself is not properly infinite. Then Theorem~\ref{theorem:main1} implies that $(X,d)$ is amenable. Now suppose that $(X,d)$ were {\em not} a properly amenable metric space. Corollary~\ref{cor:characamenable-nprpamen} shows that there would be a partition $X= Y_1\sqcup Y_2$, where $Y_1$ is a finite non-empty subset of $X$, $Y_2$ is non-amenable and $d(x,y)= \infty $ for $x\in Y_1$ and $y\in Y_2$. As in Remark~\ref{rem:Roe-direct-sum}, this would induce a direct sum decomposition $\K_\mathrm{u}(X) \cong \K_\mathrm{u}(Y_1)\oplus \K_\mathrm{u}(Y_2)$, with $\K_\mathrm{u}(Y_1)$ being finite-dimensional. In particular, $\K_\mathrm{u}(X)$ would be a finite-dimensional extension of $\K_\mathrm{u}(Y_2)$, the latter being properly infinite, again by Theorem~\ref{theorem:main1}. This would contradict our assumption.
\end{proof}

\vspace{1cm}

\paragraph{\bf Acknowledgements:}
The second-named author is partially supported by Deutsche Forschungsgemeinschaft (SFB 878). The third-named author thanks Wilhelm Winter for his kind invitation to the Mathematics 
Department of the University of M\"unster in April 2014 and March-June 2016.
Financial support was provided by the DFG through SFB 878, as well as, by a DAAD-grant during these visits. He would also like to thank the organizers
of the Thematic Program on 
{\em Abstract Harmonic Analysis, Banach and Operator Algebras} at Fields Institute in 
Toronto in May 2014 for the stimulating atmosphere. The fourth-named author is grateful to David Kerr for some very helpful suggestions. 
Part of the research was conducted during visits and workshops at Universitat Aut\`onoma de Barcelona, University of Copenhagen, University of M\"{u}nster and Institut Mittag-Leffler. 
The authors owe many thanks and great appreciation to these institutes and hosts for their hospitality.


\providecommand{\bysame}{\leavevmode\hbox to3em{\hrulefill}\thinspace}

\end{document}